\documentclass{article}

\usepackage{amsmath}  %
\usepackage{amsfonts} 
\usepackage{amssymb}  %
\usepackage{amsthm}   %

\DeclareMathOperator{\Ann}{Ann}
\DeclareMathOperator{\ran}{ran}
\newcommand{\RR}{\mathbb{R}}  
\newcommand{\CC}{\mathbb{C}}  
\newcommand{\NN}{\mathbb{N}}  
\newcommand{\DD}{\mathbb{D}}  
\newcommand{\TT}{\mathbb{T}}  
\newcommand{\EE}{\mathbb{E}}  
\newcommand{\VV}{\mathbb{V}}
\newcommand{\WW}{\mathbb{W}}
\newcommand{\XX}{\mathbb{X}}
\newcommand{\pd}{\partial}
\newcommand{\sbs}{\subset}

\newcommand{\sbse}{\subseteq}
\newcommand{\spse}{\supseteq}
\newcommand{\Hil}{\mathcal{H}} 
\newcommand{\of}{\circ} 
\newcommand{\nin}{\notin} 
\newcommand{\bksl}{\backslash}
\newcommand{\SUF}{\Rightarrow}      %

\newcommand{\mc}[1]{\mathcal{#1}}
\newcommand{\cc}[1]{\overline{#1}}
\newcommand{\pvect}[1]{\begin{pmatrix} #1 \end{pmatrix}} 
\newcommand{\wtil}[1]{\widetilde{#1}}
\newcommand{\what}[1]{\widehat{#1}}
\newcommand{\inner}[2]{ \langle #1 , #2 \rangle} 

\newcommand{\gra}{\alpha}
\newcommand{\grb}{\beta}

\newcommand{\grd}{\delta}

\newcommand{\grg}{\gamma}

\newcommand{\grj}{\theta}
\newcommand{\grk}{\kappa}
\newcommand{\grl}{\lambda}
\newcommand{\grs}{\sigma}

\newcommand{\grw}{\omega}
\newcommand{\grz}{\zeta}

\newcommand{\grD}{\Delta}

\newcommand{\grJ}{\Theta}

\newcommand{\grS}{\Sigma}

\newcommand{\grW}{\Omega}

\theoremstyle{plain}
\newtheorem{theorem}{Theorem}[section]
\newtheorem{lemma}[theorem]{Lemma}
\newtheorem{proposition}[theorem]{Proposition}
\newtheorem{corollary}[theorem]{Corollary}

\theoremstyle{definition}

\newtheorem{example}[theorem]{Example}

\theoremstyle{remark}

\title{On polynomial $n$-tuples of commuting isometries}
\author{Edward J. Timko}

\begin{document}

\maketitle

\begin{abstract}
  We extend some of the results of Agler, Knese, and McCarthy \cite{AKM} to $n$-tuples of commuting isometries for $n>2$. Let $\VV=(V_1,\dots,V_n)$ be an $n$-tuple of a commuting isometries on a Hilbert space and let $\Ann(\VV)$ denote the set of all $n$-variable polynomials $p$ such that $p(\VV)=0$. When $\Ann(\VV)$ defines an affine algebraic variety of dimension 1 and $\VV$ is completely non-unitary, we show that $\VV$ decomposes as a direct sum of $n$-tuples $\WW=(W_1,\dots,W_n)$ with the property that, for each $i=1,\dots,n$, $W_i$ is either a shift or a scalar multiple of the identity. If $\VV$ is a cyclic $n$-tuple of commuting shifts, then we show that $\VV$ is determined by $\Ann(\VV)$ up to near unitary equivalence, as defined in \cite{AKM}.
\end{abstract}

\section{Introduction}

Agler, Knese, and McCarthy \cite{AKM} prove several results concerning pairs of commuting shifts subject to a polynomial equation. Recall that $V$ is a \textit{shift} on a Hilbert space $\Hil$ if $V$ is an isometry and $\bigcap_{j=0}^\infty V^j\Hil =\{0\}$. We extend some of the results in \cite{AKM} to certain collections of $n$ commuting isometries.

We begin by setting some notation and terminology. Let $\VV=(V_1,\dots,V_n)$ be an $n$-tuple of commuting isometries on a Hilbert space $\Hil$. Generally we suppose that there are polynomials $p_1,\dots,p_k\in\CC[X_1,\dots,X_n]$ such that $p_1(\VV)=\dots=p_k(\VV)=0$ and $p_1,\dots,p_k$ determine an algebraic variety of pure dimension 1. We say that $\VV$ is \textit{completely non-unitary} if there is no non-zero subspace $\mc{K}$ of $\Hil$ that is reducing for each element of $\VV$ such that $V_i|\mc{K}$ is a unitary operator for $i=1,\dots,n$. An arbitrary $n$-tuple of commuting isometries admits a decomposition as a direct sum of an $n$-tuple of unitaries and a completely non-unitary $n$-tuple, and thus we often focus on completely non-unitary $n$-tuples.

Given a set $S$ of polynomials in $\CC[X_1,\dots,X_n]$, we denote by $Z(S)$ the variety determined by $S$; that is, $Z(S)=\{z\in\CC^n:p(z)=0\text{ for all }p\in S\}$. An ideal $\mc{I}$ in $\CC[X_1,\dots,X_n]$ is said to be \textit{radical} if it coincides with its radical; that is, $p^k\in\mc{I}$ implies $p\in\mc{I}$. It is well known (see \cite[Thm. III.4.6]{Kendig}) that a radical ideal $\mc{I}$ can be written as the irredundant intersection of a unique finite family of prime ideals. These ideals are called the \textit{prime factors of $\mc{I}$}. The \textit{annihilator} of $\VV$ is the ideal $\Ann(\VV)$ of all polynomials $p\in\CC[X_1,\dots,X_n]$ such that $p(\VV)=0$. Below are a few examples where the annihilator is non-zero.
\begin{example}\label{Intro:AKMExample}
    Let $(V_1,V_2)$ be a commuting pair of shifts so that $\dim\ker V_1^*$ and $\dim\ker V_2^*$ are finite; say $k=\dim\ker V_2^*$. Then we can represent $(V_1,V_2)$ as a pair of Toeplitz operators $(T_{\grJ},T_{\grz\cdot I_k})$ on the $\CC^k$-valued Hardy space $H^2(\DD)\otimes\CC^k$, where $\grz$ is the coordinate function on $\DD$, $I_k$ the identity operator on $\CC^k$, and $\grJ$ is a matrix valued rational inner function (see \cite{AgMc2005} for details). Then there are polynomials $P\in\CC[X_1,X_2]$ and $Q\in\CC[X_2]$ so that $\det(w I_k-\Phi(z))=P(w,z)/Q(z)$ for $w,z\in\DD$. It then follows that $P(V_1,V_2)=0$.
\end{example}

\begin{example}
    Define the $4\times 4$ matrices $\Phi_1(z),\Phi_2(z)$ by
    \[ \Phi_1(z)u=(u_3,u_4,z u_2,z^3 u_1)^t, \quad \Phi_2(z)u=(u_2,z^2u_1,u_4,z^2u_3)^t, \]
    where $u=(u_1,u_2,u_3,u_4)^t$. Then $(V_1,V_2,V_3)=(T_{\Phi_1},T_{\Phi_2},T_{\grz\cdot I_4})$ acting on $H^2(\DD)\otimes \CC^4$ is a triple of commuting shifts satisfying the equations $V_1^2=V_2V_3$ and $V_2^2=V_3^2$.
\end{example}

In general, however, the annihilator may be $\{0\}$, as is the case for the pair $(S\otimes 1,1\otimes S)$ acting on $\ell^2(\NN)\otimes\ell^2(\NN)$ where $S$ is the unilateral shift on $\ell^2(\NN)$.

In what follows, assume $\VV$ to be an $n$-tuple of commuting isometries on a Hilbert space $\Hil$. We proceed now to state the main results of this paper, and begin by providing a description of the annihilator of a completely non-unitary $n$-tuple.
\begin{quotation}
  \noindent\textbf{Theorem \ref{CharAnn:MainThm}.} Suppose $\VV$ is completely non-unitary, and let $\mc{I}\sbse \Ann(\VV)$ be an ideal such that $\dim Z(\mc{I})=1$.
    \begin{enumerate}
        \item[\rm{(1)}] $\mc{I}=\Ann(\VV)$ if and only if $\mc{I}$ is radical with prime factors $\mc{I}_1,\dots,\mc{I}_m$ such that $\bigcap_{i\neq j}\mc{I}_j\nsubseteq \Ann(\VV)$ for $j=1,\dots,m$.
        \item[\rm{(2)}] If $\mc{I}=\Ann(\VV)$, then there exist non-zero mutually orthogonal $\VV$-invariant subspaces $\Hil_1,\dots,\Hil_m$ of $\Hil$ such that $\mc{I}_j=\Ann(\VV|\Hil_j)$ for each $j\in\{1,\dots,m\}$.
    \end{enumerate}
\end{quotation}
An algebraic variety $\mc{W}$ in $\CC^k$ is said to be a \textit{distinguished variety} if
\[ \mc{W}\sbse \DD^k\cup \TT^k \cup \EE^k \]
where $\DD$ is the open unit disc in $\CC$, $\TT$ is the unit circle, $\EE$ is the exterior of the closed unit disc, and exponents indicate Cartesian powers. In the special case that $\Ann(\VV)$ is a prime ideal, the $n$-tuple $\VV$ has a particularly simple structure.
\begin{quotation}
  \noindent \textbf{Theorem \ref{PrmIdl:MainThm}.} Suppose $\VV$ is completely non-unitary. If $\Ann(\VV)$ is a prime ideal and $Z(\Ann(\VV))$ has dimension 1, then, after a permutation of coordinates, there exists an $s\in\{1,\dots,n\}$ such that 
    \begin{enumerate}
        \item[\rm{(1)}] $\VV=(V_1,\dots,V_s,\grl_{s+1}I,\dots,\grl_n I)$ where $V_1,\dots,V_s$ are shifts and $\grl_{s+1},\dots,\grl_n$ are scalars of absolute value 1; and
        \item[\rm{(2)}] $Z(\Ann(\VV))=\mc{W}\times\{(\grl_{s+1},\dots,\grl_n)\}$ for some 1-dimensional distinguished variety $\mc{W}\sbse \CC^s$.
    \end{enumerate}
\end{quotation}
When $\Ann(\VV)$ is not prime, we have the following result.
\begin{quotation}
  \noindent\textbf{Theorem \ref{Decomp:MainThm}.} Suppose $\VV$ is completely non-unitary and $Z(\Ann(\VV))$ has dimension 1. There is a subset $\mc{F}\sbse (\{0\}\cup\TT)^n$ and a collection of non-zero $\VV$-reducing subspaces $(\Hil_z)_{z\in\mc{F}}$ such that $\Hil=\bigoplus_{z\in\mc{F}}\Hil_z$ and, for each $z=(z_1,\dots,z_n)\in\mc{F}$, the following hold.
  \begin{enumerate}
    \item[\rm{(1)}] If $i\in\{1,\dots,n\}$ and $z_i=0$, then $V_i|\Hil_z$ is a shift. If $z_i\neq 0$, then $(V_i-z_i I)|\Hil_z=0$.
    \item[\rm{(2)}] Suppose exactly $s$ components of $z$ are 0, and let $\grl_{s+1},\dots,\grl_n$ denote the non-zero components of $z$. Then there is a one dimensional distinguished variety $\mc{W}$ in $\CC^s$ such that, after a permutation of coordinates, $Z(\Ann(\VV|\Hil_z))=\mc{W}\times\{(\grl_{s+1},\dots,\grl_n)\}$.
  \end{enumerate}
\end{quotation}
Recall that a subset $S$ of $\Hil$ is said to be \textit{cyclic} for $\Hil$ if the set $\{p(\VV)h:p\in\CC[X_1,\dots,X_n],h\in S\}$ is total in $\Hil$. When $\VV$ is an $n$-tuple of shifts for which $\dim Z(\Ann(\VV))=1$, there is a close relationship between finite multiplicity and the presence of a finite cyclic set.
\begin{quotation}
  \noindent\textbf{Theorem \ref{FinMult:Equiv}.} If $\VV$ is an $n$-tuple of shifts and $Z(\Ann(\VV))$ has dimension $1$, then the following assertions are equivalent.
    \begin{enumerate}
        \item[\rm{(1)}] There is a finite cyclic set for $\VV$.
        \item[\rm{(2)}] $V_i$ has finite multiplicity for each $i\in\{1,\dots,n\}$.
        \item[\rm{(3)}] $V_i$ has finite multiplicity for some $i\in\{1,\dots,n\}$.
    \end{enumerate}
\end{quotation}
In fact, when $V_1,$ $\dots,$ $V_n$ each have finite multiplicity, the hypothesis that $\dim Z(\Ann(\VV))=1$ is unnecessary;
\begin{quotation}
    \noindent\textbf{Proposition \ref{FinMult:DimForFree}.} If $\VV$ is an $n$-tuple of shifts such that $V_i$ has finite multiplicity for $i=1,\dots,n$, then $Z(\Ann(\VV))$ has dimension $1$.
\end{quotation}
We remark that this is essentially a corollary of Example \ref{Intro:AKMExample}.

Two $n$-tuples of commuting isometries, say $\VV$ on a Hilbert space $\Hil$ and $\WW$ on a Hilbert space $\mc{K}$, are said to be \textit{nearly unitarily equivalent} if $\VV$ is unitarily equivalent to the restriction of $\WW$ to a finite codimensional $\WW$-invariant subspace and, similarly, $\WW$ is unitarily equivalent to the restriction of $\VV$ to a finite codimensional $\VV$-invariant subspace.
\begin{quotation}
  \noindent\textbf{Theorem \ref{FinalSec:MainThm}.} Suppose $\VV$ and $\WW$ are cyclic $n$-tuples of commuting shifts. If $\Ann(\VV)=\Ann(\WW)$ and $\dim Z(\Ann(\VV))=1$, then $\VV$ and $\WW$ are nearly unitarily equivalent.
\end{quotation}
Many of the results listed here have direct analogues in \cite{AKM}. In particular Theorem \ref{Decomp:MainThm} is based on \cite[Thm. 1.20]{AKM}, Theorem \ref{FinMult:Equiv} is based of the proof \cite[Lem. 1.11]{AKM}, and Theorem \ref{FinalSec:MainThm} is analogous to \cite[Thm. 3.3]{AKM}.

I would like to given special thanks Hari Bercovici for his many illuminating comments and criticisms during the development of this paper. I would also like to thank Carl C. Cowen, Chris Judge, and Norm Levenberg for their various conversations with me.

\section{Preliminary Material}

In this section we set down some notation and state a few results that we use. Throughout this paper $n$ denotes a positive integer, and $\VV=(V_1,\dots,V_n)$ an $n$-tuple of commuting isometries on a Hilbert space $\Hil$. A subspace $\mc{K}$ of $\Hil$ is \textit{$\VV$-invariant} if $\mc{K}$ is an invariant subspace for each $V_i$, while $\mc{K}$ is \textit{$\VV$-reducing} if $\mc{K}$ is a reducing for each $V_i$. In either case, we set $\VV|\mc{K}=(V_1|\mc{K},\dots,V_n|\mc{K})$. With this notation, the $n$-tuple $\VV$ is completely non-unitary if there is no non-zero $\VV$-reducing subspace $\mc{K}$ such that $V_i|\mc{K}$ is unitary for each $i$. If there is no non-zero $\VV$-reducing subspace $\mc{K}$ such that $V_i|\mc{K}$ is unitary for some $i$, then we say that $\VV$ is \textit{pure}.

\begin{proposition}[{\cite{Suciu}}]
    Given a pair of commuting isometries $(V_1,V_2)$ on a Hilbert space $\Hil$, 
    \[ (V_1,V_2)=(S_1,S_2)\oplus (U_1,T_2)\oplus (T_1,U_2)\oplus (W_1,W_2) \]
where $U_1,U_2,W_1,W_2$ are unitary, $T_1,T_2$ are shifts, and $(S_1,S_2)$ is pure.
\end{proposition}
Here we are using the following notation. If, for each $i\in I$, $\mathbb{A}^{(i)}$ is an $n$-tuple on operators acting a Hilbert space $\Hil_i$, then
\[ \bigoplus_{i\in I}\mathbb{A}^{(i)}=\bigg(\bigoplus_{i\in I}A^{(i)}_1,\dots,\bigoplus_{i\in I}A^{(i)}_n\bigg) \]
is an $n$-tuple of operators acting on the Hilbert space $\bigoplus_{i\in I}\Hil_i$.

We also require some notation for polynomial functions. Denote by $\CC[\XX] = \CC[X_1,\dots,X_n]$ the ring of $n$-variable polynomials, and set $\XX=(X_1,\dots,X_n)$. As is common practice, we frequently identify a given $p\in \CC[\XX]$ with the corresponding map $z\mapsto p(z)$ on $\CC^n$. Let $\NN_0$ denote the set of non-negative integers $\{0,1,2,\dots\}$, and set $\XX^\gra=X_1^{\gra_1}\cdots X_n^{\gra_n}$
whenever $\gra=(\gra_1,\dots,\gra_n)\in\NN_0^n$. Given an $n$-tuple $\VV$ of commuting isometries, we set $p(\VV)=p(V_1,V_2,\dots,V_n)$ for each $p\in\CC[\XX]$, $\VV^\gra=V_1^{\gra_1}\cdots V_n^{\gra_n}$ for each $\gra\in\NN^n_0$, and $\VV^*=(V_1^*,\dots,V_n^*)$. The ideal $\{p\in\CC[\XX]:p(\VV)=0\}$ is called the \textit{annihilator} of $\VV$ and is denoted by $\Ann(\VV)$. Whenever $L$ is a linear subspace of $\CC[\XX]$, we set $L(\VV)=\{p(\VV):p\in L\}$, and when $S$ is a subset of $\Hil$ we define $L(\VV)S$ to be the (algebraic) linear span of $\{p(\VV)u:p\in L,u\in S\}$.

\begin{proposition}[{\cite{Suciu}}]
    Let $\VV$ be an $n$-tuple of commuting isometries on a Hilbert space $\Hil$. Then $\bigcap_{\gra\in\NN_0^n}\VV^\gra\Hil$ is a $\VV$-reducing subspace and the largest $\VV$-invariant subspace of $\Hil$ on which $\VV$ is an $n$-tuple of unitary operators. Thus $\VV$ is completely non-unitary if and only if $\bigcap_{\gra\in\NN_0^n}\VV^\gra\Hil=\{0\}$.
\end{proposition}
As we can always isolate the completely unitary part of a given $n$-tuple, we generally assume our $n$-tuple is completely non-unitary. If $\VV$ is completely non-unitary and $\mc{K}$ is a $\VV$-invariant subspace of $\Hil$, then $\bigcap_{\gra\in\NN_0^n}\VV^\gra\mc{K}\sbse \bigcap_{\gra\in\NN_0^n}\VV^\gra\Hil=\{0\}$. Thus the restriction of a completely non-unitary $n$-tuple is again completely non-unitary. We also note the following. Suppose $A$ is an operator on $\Hil$ commuting with $\VV$ and $\ker A=\{0\}$. If $\VV|\cc{\ran A}$ is completely non-unitary, then $\VV$ is completely non-unitary.

\begin{proposition}[{\cite{Ito}}]
    Given an $n$-tuple of commuting isometries $\VV$ on a Hilbert space $\Hil$, there exists an $n$-tuple of commuting unitaries $\wtil{\VV}=(\wtil{V}_1,\dots,\wtil{V}_n)$ on a Hilbert space $\wtil{\Hil}$ such that
    \begin{enumerate}
        \item[\rm{(1)}] $\Hil$ is a $\wtil{\VV}$-invariant subspace of $\wtil{\Hil}$;
        \item[\rm{(2)}] $\wtil{\VV}|\Hil=\VV$; and
        \item[\rm{(3)}] $\wtil{\Hil}=\bigvee_{\gra\in\NN^n}(\wtil{\VV})^{*\gra}\Hil$.
    \end{enumerate}
    These properties determine $\wtil{\VV}$ up to unitary equivalence. We call $\wtil{\VV}$ the \textit{minimal unitary extension} of $\VV$. Furthermore, if $A$ is an operator on $\Hil$ commuting with $\VV$, then there exists a unique operator $\wtil{A}$ on $\wtil{\Hil}$ that commutes with $\wtil{\VV}$ such that $\wtil{A}|\Hil=A$ and $\|\wtil{A}\|=\|A\|$. We call $\wtil{A}$ the \textit{canonical extension of $A$}.
\end{proposition}
 For any $p\in\CC[\XX]$, we have
    \[ \|p(\VV)\|=\|p(\wtil{\VV})\|. \]
 It follows that $\Ann(\VV)$ is a radical ideal; that is, if $p^k\in\Ann(\VV)$ for some positive integer $k$, then $p\in\Ann(\VV)$.

The following two propositions are based on results from Chapter V and VI of \cite{NFBK}.
\begin{proposition}\label{Fact:OpValFunc}
    Let $(V_1,V_2)$ be a pair of commuting shifts on a Hilbert space $\Hil$.
    \begin{enumerate}
        \item[\rm{(1)}] There is a Hilbert space $\mc{E}$ of dimension $\dim\ker V_2^*$ and an $\mc{L}(\mc{E})$-valued inner function $\grJ$ on $\DD$ such that $(V_1,V_2)$ is unitarily equivalent to pair of Toeplitz operators $(T_{\grJ},T_{\grz I})$ acting on $H^2(\DD)\otimes \mc{E}$. 
        \item[\rm{(2)}] In the case that both $V_1$ and $V_2$ have finite multiplicity, $\mc{E}$ is finite dimensional and $\grJ$ is a matrix valued inner function with entries consisting of rational functions.
    \end{enumerate}
\end{proposition}
\noindent We remark that Proposition \ref{Fact:OpValFunc}(2) is a consequence of \cite[Thm. VI.3.1]{NFBK} and \cite[Prop. VI.3.2]{NFBK}.

\begin{proposition}[{\cite[Thm. 3.1]{BDF}}] \label{Fact:BDF}
    Let $V$ be a shift on a Hilbert space $\mc{E}$, and let $A$ be a contraction on $\mc{E}$ commuting with $V$. Letting $\wtil{V}$ denote the minimal unitary extension of $V$ and $\wtil{A}$ the canonical extension of $A$, set $D=(1-(\wtil{A})^*\wtil{A})^{1/2}$ and $\mc{D}=\cc{\ran D}$. We observe that $\mc{D}$ is a reducing subspace for $\wtil{V}$. On the Hilbert space $\Hil\oplus\left[\bigoplus_{i=0}^\infty\mc{D}\right]$, we define operators $W_1$ and $W_2$ by
    \begin{eqnarray*}
        W_1(h,d_0,d_1,\dots) & = & (Vh,\wtil{V}d_0,\wtil{V}d_1,\dots) \\
        W_2(h,d_0,d_1,\dots) & = & (Ah,Dh,d_0,d_1,\dots),
    \end{eqnarray*}
   where $h\in\Hil$ and $d_0,d_1,\dots\in\mc{D}$. Then $\WW_A=(W_1,W_2)$ is a pair of commuting isometries. Given a pair $\WW'$ of commuting isometries on a Hilbert space $\Hil'$, the following assertions are equivalent.
   \begin{enumerate}
       \item[\rm{(1)}] There is no $\WW'$-reducing subspace $\mc{K}'$ of $\Hil'$ such that $W_1'|\mc{K}'$ is unitary.
       \item[\rm{(2)}] There exists a contraction $A$ commuting with a shift $V$ such that $\WW'$ is unitarily equivalent to $\WW_A$.
   \end{enumerate}
   Additionally, the pair $(V,A)$ is unique up to unitary equivalence, and thus it is called the \textit{characteristic pair} of $\WW'$.
\end{proposition}

We require a few results from elementary algebraic geometry, for which we refer the reader to \cite{CoxEtAl} and \cite{Kendig}. As we only concern ourselves with affine algebraic varieties in $\CC^\ell$, for some positive integer $\ell$, we generally drop the adjectives `affine' and `algebraic'. The following are some results we use repeatedly.

\begin{proposition}[{\cite[Thm. III.4.6]{Kendig}}] If $\mc{I}$ is a (non-zero) radical ideal of $\CC[\XX]$, then there is a unique finite collection of prime ideals $\mc{I}_1,\dots,\mc{I}_m$ such that $\bigcap_{i=1}^m\mc{I}_i=\mc{I}$ and $\bigcap_{i\neq j}^m\mc{I}_i \supsetneq \mc{I}$ for $j=1,\dots,m$. 
\end{proposition}
In other words, $\mc{I}$ is the \textit{irredundant} intersection of the prime ideals $\mc{I}_1,\dots,\mc{I}_m$; we call these the \textit{prime factors} of $\mc{I}$. Similarly, if $\mc{V}$ is a variety in $\CC^n$, then there exists a unique collection of irreducible varieties $\mc{V}_1,\dots,\mc{V}_m$ so that $\mc{V}=\bigcup_{i=1}^m\mc{V}_i$ and $\mc{V}\supsetneq \bigcup_{j\neq i}\mc{V}_j$ for every $i\in\{1,\dots,m\}$. We call $\mc{V}_1,\dots,\mc{V}_m$ the \textit{irreducible components} of $\mc{V}$.

\begin{proposition}[{\cite[Thm. IV.3.1]{Kendig}}]\label{Fact:CodimIneq}
    If $\mc{V}$ and $\mc{W}$ are irreducible varieties in $\CC^n$ and $\mc{V}\cap\mc{W}\neq \emptyset$, then
    \[ \dim\mc{V} + \dim\mc{W}\leq n+\dim(\mc{V}\cap\mc{W}). \]
\end{proposition}

\begin{proposition}[{\cite[Thm. IV.2.15]{Kendig}}]\label{Fact:ProperSubvar} If $\mc{V}$ is an irreducible variety in $\CC^n$ and $\mc{W}$ is a proper subvariety, then $\dim\mc{W}<\dim\mc{V}$.
\end{proposition}
Thus, if $\mc{V}$ is an irreducible variety of dimension 1 and $\mc{W}$ is any variety in $\CC^n$, then $\mc{W}\cap\mc{V}$ is either $\mc{V}$ or a finite set.

\begin{proposition}[{\cite[Cor. 9.5.4]{CoxEtAl}}]\label{Fact:TwoVarsExist} Given an ideal $\mc{I}\sbse\CC[\XX]$, the dimension of $Z(\mc{I})$ is equal to the largest integer $r$ for which there exist $r$ distinct elements $i_1,\dots,i_r\in\{1,\dots,n\}$ such that
        \[ \mc{I}\cap \CC[X_{i_1},\dots,X_{i_r}]=\{0\}. \]
\end{proposition}
In particular, if $\dim Z(\mc{I})=1$, then $\CC[X_i,X_j]\cap\mc{I}$ contains a non-zero element whenever $i\neq j$.

\begin{proposition}[{\cite[Thm. 5.3.6]{CoxEtAl}}] Given an ideal $\mc{I}\sbse\CC[\XX]$, the variety $Z(\mc{I})$ is a finite set if and only if $\CC[\XX]/\mc{I}$ has finite linear dimension.
\end{proposition}

\noindent A corollary of this result is the following.

\begin{corollary}\label{Intro:FctrsLem}
  Suppose $\mc{I}_1,\dots,\mc{I}_m$ are the prime factors of a radical ideal, $m>1$, and $\dim Z(\mc{I}_i)=1$ for $i=1,\dots,m$. Then $\CC[\XX]/\mc{J}$ has finite linear dimension, where $\mc{J}=\sum_{i=1}^m \bigcap_{j\neq i}\mc{I}_j$.
\end{corollary}
\begin{proof}
  We observe that
  \[ Z(\mc{J})=\bigcap_{i=1}^m\bigcup_{j\neq i}Z(\mc{I}_j)=\bigcup_{\substack{1\leq i,j\leq m \\ i\neq j}}Z(\mc{I}_i)\cap Z(\mc{I}_j). \]
  Because the collection $\mc{I}_1,\dots,\mc{I}_m$ is irredundant, $Z(\mc{I}_i)\cap Z(\mc{I}_j)$ is a finite set when $i\neq j$, and thus $Z(\mc{J})$ is a finite set.
\end{proof}

 Before concluding this section, we make the following observation about distinguished varieties.

\begin{proposition}
    Suppose $\mc{V}$ is a distinguished variety in $\CC^n$ and each irreducible component of $\mc{V}$ meets $\cc{\DD^n}$. Then $\mc{V}$ has dimension at most $1$.
\end{proposition}
\begin{proof}
    It suffices to consider the case where $\mc{V}$ is irreducible. We fix $z\in \mc{V}\cap\cc{\DD^n}$ and set $H=\{w\in\CC^n:w_n=z_n\}$. Because $\mc{V}$ is distinguished, $H\cap\mc{V}\sbse \cc{\DD^n}$. Since this is a bounded and therefore compact variety, it follows from \cite[Prop. I.3.1]{Chirka} that $H\cap\mc{V}$ is a finite set. By Proposition \ref{Fact:CodimIneq},
    \[ \dim\mc{V}+\dim H \leq n+\dim(H\cap\mc{V}). \]
    As $\dim(H\cap\mc{V})=0$ and $\dim H=n-1$, we have that $\dim\mc{V}\leq 1$.
\end{proof}

\section{Basic Properties of $\Ann(\VV)$}\label{Sec:CharAnn}

Throughout this section, assume that $\VV$ is a completely non-unitary $n$-tuple of commuting isometries. The main result is Theorem \ref{CharAnn:MainThm}, which allows us to recover $\Ann(\VV)$ from a given ideal $\mc{I}\sbse\Ann(\VV)$ with $\dim Z(\mc{I})=1$.

\begin{lemma}\label{CharAnn:NZDim}
    No component of $Z(\Ann(\VV))$ has dimension $0$.
\end{lemma}
\begin{proof}
  Suppose $\Ann(\VV)\neq \{0\}$, and let $\mc{I}_1,\dots,\mc{I}_m$ be the prime factors of $\Ann(\VV)$. We note that $\mc{I}_2\cap\dots\cap\mc{I}_m$ is strictly larger than $\Ann(\VV)$ and therefore it cannot annihilate $\VV$. Suppose that $Z(\mc{I}_1)$ has dimension 0, and thus is generated by $X_1-\grl_1,\dots,X_n-\grl_n$ for some $\grl_1,\dots,\grl_n\in\CC$. Let $\Hil_1$ be the closure of $(\mc{I}_2\cap\dots\cap\mc{I}_m)(\VV)\cdot\Hil$ and note that $\Hil_1\neq \{0\}$. Since $\VV|\Hil_1$ is annihilated by $\mc{I}_1$, we have
  \[ \VV|\Hil_1 = (\grl_1,\dots,\grl_n). \]
  But this would mean that $\VV$ is not completely non-unitary.
\end{proof}

\begin{lemma}\label{CharAnn:Prime}
  If $\mc{I}\sbse \Ann(\VV)$ is a prime ideal and $Z(\mc{I})$ has dimension $1$, then $\Ann(\VV)=\mc{I}$.
\end{lemma}
\begin{proof}
  We observe that $Z(\Ann(\VV))\sbse Z(\mc{I})$ and that the latter is an irreducible variety of dimension 1. Since $\VV$ is completely non-unitary, it follows from Lemma \ref{CharAnn:NZDim} that every irreducible component of $Z(\Ann(\VV))$ must have dimension 1. By Proposition \ref{Fact:ProperSubvar}, $\Ann(\VV)=\mc{I}$.
\end{proof}

Given $p\in\CC[\XX]$, set $\grd(p)=(d_1,\dots,d_n)$, where $d_i$ is the degree of $X_i$ in $p$. We define $\rho:\CC[\XX]\to\CC[\XX]$ by setting
    \[ \rho(p)(z)=z^{\grd(p)}\cc{p(1/\cc{z})}, \quad z\in (\CC\bksl\{0\})^n, \]
    where $1/\cc{z}=(1/\cc{z}_1,\dots,1/\cc{z}_n)$. If we write $p=\sum_{\gra\geq 0}c_\gra \XX^\gra$, then
    \[ \rho(p)=\sum_{\gra\geq 0}\cc{c_\gra}\XX^{\grd(p)-\gra}. \]
    Thus, when $\WW$ is an $n$-tuple of commuting isometries and $p$ a polynomial, we have
      \[ \rho(p)(\WW)=p(\WW)^*\WW^{\grd(p)},  \]
    and therefore $\rho(\Ann(\WW))\sbse \Ann(\WW)$. When $\mc{I}$ is an ideal in $\CC[\XX]$ with the property that $\rho(\mc{I})\sbse \mc{I}$, then $Z(\mc{I})$ has a certain reflection property. Specifically, if $z\in Z(\mc{I})$ and no coordinate of $z$ is 0, then $1/\cc{z}\in Z(\mc{I})$. A final remark concerning $\rho$ is the following lemma, which is based on \cite[Lem. 2.2]{AKM}.

\begin{lemma}\label{CharAnn:Ortho}
    Let $\WW$ be an $n$-tuple of commuting isometries on a Hilbert space $\Hil$, and let $\mc{I},\mc{J}$ be ideals in $\CC[\XX]$ such that $\rho(\mc{I})\sbse \mc{I}$ and $\mc{I}\cdot\mc{J}\sbse \Ann(\WW)$. Then $\cc{\mc{I}(\WW)\Hil}$ and $\cc{\mc{J}(\WW)\Hil}$ are orthogonal subspaces of $\Hil$.
\end{lemma}
\begin{proof}
    Let $p\in\mc{I}$ and $q\in\mc{J}$. Then
    \[ p(\WW)^*q(\WW)=\WW^{*\grd(p)}p(\WW)^*\WW^{\grd(p)}q(\WW)=\WW^{*\grd(p)}(\rho(p)\cdot q)(\WW)=0. \]
    Thus $\inner{p(\WW)h}{q(\WW)h'}=0$ for each $p\in\mc{I}$, $q\in\mc{J}$, and $h,h'\in\Hil$.
\end{proof}

\begin{theorem}\label{CharAnn:MainThm}
    Let $\mc{I}\sbse \Ann(\VV)$ be an ideal such that $\dim Z(\mc{I})=1$.
    \begin{enumerate}
        \item[\rm{(1)}] $\mc{I}=\Ann(\VV)$ if and only if $\mc{I}$ is radical and $\bigcap_{i\neq j}\mc{I}_j\nsubseteq \Ann(\VV)$ for every $j\in\{1,\dots,m\}$, where $\mc{I}_1,\dots,\mc{I}_m$ are the prime factors of $\mc{I}$.
        \item[\rm{(2)}] If $\mc{I}=\Ann(\VV)$, then there exist non-zero, mutually orthogonal $\VV$-invariant subspaces $\Hil_1,\dots,\Hil_m$ of $\Hil$ such that $\mc{I}_j=\Ann(\VV|\Hil_j)$ for each $j\in\{1,\dots,m\}$.
    \end{enumerate}
\end{theorem}
\begin{proof}
    (1) If $\mc{I}=\Ann(\VV)$, then $\mc{I}$ is radical and so $\Ann(\VV)\subsetneq\bigcap_{i\neq j}\mc{I}_i$ for $j=1,\dots,m$.
    
    Conversely, suppose $\mc{I}$ is a radical ideal and that $\what{\mc{I}}_j=\bigcap_{i\neq j}\mc{I}_i$ does not annihilate $\VV$ for any $j\in\{1,\dots,n\}$. We define $\Hil_j$ to be the closure of $\what{\mc{I}}_j(\VV)\Hil$; this is non-zero and $\mc{I}_j\sbse \Ann(\VV|\Hil_j)$. By Lemma \ref{CharAnn:Prime}, it follows that $\mc{I}_j=\Ann(\VV|\Hil_j)$ and so
    \[ \Ann(\VV)\sbse \Ann(\VV|\Hil_j)=\mc{I}_j. \]
    Since this holds for every $j$, we have $\Ann(\VV)=\mc{I}$.
    
    (2) Given $i\in\{1,\dots,m\}$, we note that $\mc{I}_i=\Ann(\VV|\Hil_i)$ implies $\rho(\mc{I}_i)\sbse \mc{I}_i$, and therefore $\rho(\what{\mc{I}}_j)\sbse \what{\mc{I}}_j$ for each $j$. Since $\what{\mc{I}}_i\cdot\what{\mc{I}}_j\sbse \Ann(\VV)$ when $j\neq i$, the theorem now follows from Lemma \ref{CharAnn:Ortho}.
\end{proof}

Suppose $\mc{I}\sbse \Ann(\VV)$ is an ideal determining a variety of dimension 1. In order to show that $\Ann(\VV)$ can be recovered from $\mc{I}$, we first observe that $\sqrt{\mc{I}}\sbse \sqrt{\Ann(\VV)}=\Ann(\VV)$. Let $\mc{I}_1,\dots,\mc{I}_m$ be the prime factors of $\sqrt{\mc{I}}$ and let $J\sbse \{1,\dots,n\}$ be a subset of minimal cardinality for which $\bigcap_{j\in J}\mc{I}_j\sbse \Ann(\VV)$. Then $\bigcap_{j\in J}\mc{I}_j$ is radical with prime factors $\{\mc{I}_j:j\in J\}$, and $Z(\bigcap_{j\in J}\mc{I}_j)$ has dimension 1. Theorem \ref{CharAnn:MainThm} now asserts that $\Ann(\VV)=\bigcap_{j\in J}\mc{I}_j$.


\section{The structure of $\VV$ when $\Ann(\VV)$ is prime}

The purpose of this section is to establish Theorem \ref{PrmIdl:MainThm}, which is Theorem \ref{Decomp:MainThm} in the special case that $\Ann(\VV)$ is a prime ideal of $\CC[\XX]$. When $Z(\Ann(\VV))$ has dimension 1, $\Ann(\VV)$ contains a irreducible element of $\CC[X_i,X_n]$ for $i=1,\dots,n-1$. This fact allows us to use results for pairs of commuting isometries to extract information in the $n$-variable case.

We say that $p\in\CC[X_1,X_2]$ \textit{has no single variable factors} if there is no non-constant $q\in \CC[X_1]\cup \CC[X_2]$ so that $q$ divides $p$.

\begin{lemma}\label{PrmIdl:NoSngVar}
    Let $\VV$ be a pair of commuting isometries on a Hilbert space $\Hil$ and suppose that $p\in\CC[X_1,X_2]\bksl\{0\}$ satisfies $p(\VV)=0$. If $p$ has no single variable factors, then
    \begin{equation}\label{PrmIdl:Eq1}
      \bigcap_{j=0}^\infty V_1^j\Hil=\bigcap_{j=0}^\infty V_2^j\Hil.
    \end{equation}
    Thus the space $\mc{K}=\bigcap_{j=0}^\infty V_1^j\Hil$ is $\VV$-reducing, $\VV|\mc{K}$ is a pair of unitary operators, and $\VV|\mc{K}^\bot$ is a pair of shifts.
\end{lemma}
\begin{proof}
    As seen before,
    \[ (V_1,V_2)=(S_1,S_2)\oplus (U_1,T_2)\oplus (T_1,U_2)\oplus (R_1,R_2) \]
    where $(S_1,S_2)$ is pure, $U_1,U_2,R_1,R_2$ are unitary, and $T_1,T_2$ are shifts. Let $\Hil_0,\Hil_1,\Hil_2$, and $\Hil_{1,2}$ denote the corresponding reducing subspaces. We observe that both $(S_1,S_2)$ and $(S_2,S_1)$ have the property that there is no reducing subspace for the pair on which the first component is unitary. Thus there is a characteristic pair associated to each. 
    
    Let $(V,A)$ be the characteristic pair associated with $(S_1,S_2)$. Preserving the notation used in Proposition \ref{Fact:BDF}, we set $(W_1,W_2)=\WW_A$ and note that $p(\WW_A)=0$ implies
    \[ 0=p(\WW_A)(0,d_0,0,0,\dots)=(0,p_0(\wtil{V})d_0,p_1(\wtil{V})d_0,\dots) \]
    where $p(X_1,X_2)=\sum_{j=0}^N p_j(X_1)X_2^j$ and $d_0\in\mc{D}$ is arbitrary. Fix $j_0\in\{0,\dots,N\}$ so that $p_{j_0}\neq 0$. Since $V$ is a shift, $p_{j_0}(\wtil{V})$ is one-to-one and therefore $d_0=0$. Thus $\mc{D}=\{0\}$, and so $W_1=V$ and $W_2=A$. In other words, $W_1$, and hence $S_1$, is a shift. Applying the same argument to the characteristic pair for $(S_2,S_1)$ shows that $S_2$ is a shift.
    
   We also have $\sum_{j=0}^N p_j(U_1)T_2^j=0$. Since $T_2$ is a shift commuting with $U_1$, it follows that $p_j(U_1)=0$ for $j=0,1,\dots,N$. Indeed, let $f\in\ker T_2^*$ and note that $U_1$ commutes with $T_2^*$. Thus $p_N(U_1)f=T^{*N}_2\left(\sum_{j=0}^N p_j(U_1)T_2^j\right)f=0$, and so $p_N(U_1)T^j_2f=0$ for every $j\in\NN_0$. Since $\ker T_2^*$ is a cyclic set for $T_2$, it follows that $p_N(U_1)=0$. By the obvious induction argument, we see that $p_0(U_1)=\dots=p_{N-1}(U_1)=0$. Thus the spectrum of $U_1$ is finite, and if $q_1$ is the minimal annihilating polynomial of $U_1$, then $q_1$ divides $p_0,p_1,\dots,p_N$. However, $p$ has no single variable factors, and therefore $q_1$ is constant. This can only occur if the spectrum of $U_1$ is empty, and thus $\Hil_1=\{0\}$. In a similar way, we deduce that $\Hil_2=\{0\}$.
   
   We proceed to prove \eqref{PrmIdl:Eq1}. Recall that $\bigcap_{j=0}^\infty V_1^j\Hil$ is the largest $V_1$-reducing subspace on which $V_1$ is unitary. Thus, $\Hil_{1,2}$ is contained in $\bigcap_{j=0}^\infty V_1^j\Hil$. We also note that $\left(\bigcap_{j=0}^\infty V_1^j\Hil\right)^\bot$ is the largest $V_1$-reducing subspace on which $V_1$ is a shift. Thus $\Hil_{1,2}=\bigcap_{j=0}^\infty V_1^j\Hil$. In a similar manner we find that $\Hil_{1,2}=\bigcap_{j=0}^\infty V_2^j\Hil$.
\end{proof}

We remark that the `no single variable factors' hypothesis is only used to eliminate the mixed summands $(U_1,T_2)$ and $(T_1,U_2)$. In the case where $p$ does have single variable factors, then the middle two summands may remain, but in this case $U_1$ and $U_2$ will have finite spectra.

When $\VV$ is an $n$-tuple such that the pairs $(V_1,V_n),$ $\dots,$ $(V_{n-1},V_n)$ are algebraic, the preceding lemma implies the following. 

\begin{corollary}\label{PrmIdl:NSVCor}
  Let $\VV$ be an $n$-tuple of commuting isometries on a Hilbert space $\Hil$, and suppose there are polynomials $p_1(X_1,X_n),\dots,p_{n-1}(X_{n-1},X_n)\in\Ann(\VV)$, each of which has no single variable factors. Then $\mc{K}=\bigcap_{j=0}^\infty V_n^j\Hil$ is a $\VV$-reducing subspace of $\Hil$ such that $\VV|\mc{K}^\bot$ is an $n$-tuple of shifts and $\VV|\mc{K}$ is an $n$-tuple of unitaries.
\end{corollary}

We recall from \cite{AKM} that a polynomial $p\in\CC[X_1,X_2]$ is \textit{inner toral} if $\{z\in\CC^2:p(z_1,z_2)=0\}$ is a distinguished variety in $\CC^2$. When a given ideal in $\CC[\XX]$ has a ``enough'' inner toral two-variable polynomials, we can conclude that $Z(\mc{I})$ is a distinguished variety. More precisely, we have the following.

\begin{lemma}\label{PrmIdl:Toral}
    Let $\mc{I}$ be an ideal in $\CC[\XX]$, and suppose that there is an inner toral polynomial $p_j(X_j,X_n)\in\mc{I}$ for each $j\in\{1,\dots,n-1\}$. Then $Z(\mc{I})$ is a distinguished variety.
\end{lemma}
\begin{proof}
    Let $\mc{J}$ denote the ideal generated by $p_1(X_1,X_n),\dots,p_{n-1}(X_{n-1},X_n)$ and note that $Z(\mc{I})\sbse Z(\mc{J})$. Thus, it suffices to prove that $Z(\mc{J})$ is distinguished. Fix $w\in Z(\mc{J})$ and note that we have $w_n\in K$ for some $K\in\{\DD,\TT,\EE\}$. As $p_1(w_1,w_n)=0,\dots, p_{n-1}(w_{n-1},w_n)=0$ and $p_1,\dots,p_{n-1}$ are inner toral, it follows that each coordinate of $w$ is in $K$.
\end{proof}

\begin{theorem}\label{PrmIdl:MainThm}
    Let $\VV$ be a completely non-unitary $n$-tuple of commuting isometries. If $\Ann(\VV)$ is a prime ideal and $Z(\Ann(\VV))$ has dimension $1$, then there exists an $s\in\{1,\dots,n\}$ such that, after a permutation of coordinates, 
    \begin{enumerate}
        \item[\rm{(1)}] $\VV=(V_1,\dots,V_s,\grl_{s+1}I,\dots,\grl_n I)$ where $V_1,\dots,V_s$ are shifts and $\grl_{s+1},\dots,\grl_n$ are scalars of absolute value $1$; and
        \item[\rm{(2)}] $Z(\Ann(\VV))=\mc{W}\times\{(\grl_{s+1},\dots,\grl_n)\}$ for some $1$-dimensional distinguished variety $\mc{W}\sbse \CC^s$.
    \end{enumerate}
\end{theorem}
\begin{proof}
  For each $i\in\{1,\dots,n\}$, let $\mc{Q}_i$ be the set of irreducible polynomials in $\CC[X_i]\cap\Ann(\VV)$. After some permutation of coordinates, there is an $s\in\{0,1,\dots,n\}$ so that $\mc{Q}_i=\emptyset$ if $i\leq s$ while $\mc{Q}_i$ is non-empty if $i>s$. In fact, when $i>s$, there is a $\grl_i\in\CC$ such that $X_i-\grl_i\in\mc{Q}_i$, and thus $V_i=\grl_i I$.
  
  As $\VV$ is completely non-unitary, $s>0$. In the case that $s=1$, we have $\VV=(V_1,\grl_2 I,\dots,\grl_n I)$, and so $V_1$ is a shift and $Z(\Ann(\VV))$ is $\CC\times\{(\grl_2,\dots,\grl_n)\}$.
  
  Now we suppose $s>1$. For each $i<s$, it follows from Proposition \ref{Fact:TwoVarsExist} that there is a non-zero $p_i(X_i,X_s)\in \Ann(\VV)$. Since $\Ann(\VV)$ is prime, we may and shall assume that $p_i$ is irreducible. Since $\mc{Q}_i=\mc{Q}_s=\emptyset$, $p_i$ is non-constant in both variables. Corollary \ref{PrmIdl:NSVCor} now implies that $V_1,\dots,V_s$ are shifts.
  
   We easily see that $Z(\Ann(\VV))=\mc{W}\times\{(\grl_{s+1},\dots,\grl_n)\}$, where
  \[ \mc{W}=\{z\in\CC^s: p(z)=0\text{ for every }p\in \Ann(\VV)\cap\CC[X_1,\dots,X_s]\}. \]
  It follows from \cite[Thm. 1.20]{AKM} that, for each $i<s$, there is a inner toral polynomial $q_i(X_i,X_s)\in\Ann(\VV)$ which divides any element of $\CC[X_i,X_s]\cap\Ann(\VV)$. As $p_i$ is irreducible, it follows that $q_i$ is proportional to $p_i$ and thus that $p_i$ is inner toral. Lemma \ref{PrmIdl:Toral} now implies that $\Ann(\VV)\cap\CC[X_1,\dots,X_s]$ determines a distinguished variety in $\CC^s$.
\end{proof}

\section{A Decomposition Theorem}\label{Sec:Decomp}

We can now describe the structure of a completely non-unitary $n$-tuple of commuting isometries when $Z(\Ann(\VV))$ has dimension 1. To accomplish this, we first decompose $\VV$ into a direct sum based on the individual eigenvalues of $V_1,\dots,V_n$. This decomposition is easily derived, but we provide the details for completeness.

\begin{lemma}\label{Decomp:EgnLem1}
  Let $(V_1,V_2)$ be a pair of commuting isometries on a Hilbert space $\Hil$. Given an eigenvector $\grl$ of $V_1$, let $\mc{E}_\grl$ denote the corresponding eigenspace. Then
  \begin{enumerate}
    \item[\rm{(1)}] $\mc{E}_\grl$ is a $V_2$-reducing subspace; and
    \item[\rm{(2)}] $\mc{E}_\grl\bot\mc{E}_\eta$ if $\grl$ and $\eta$ are distinct eigenvalues of $V_1$.
  \end{enumerate}
\end{lemma}
\begin{proof}
  (2) is a triviality, so we only prove (1). Certainly $\mc{E}_\grl$ is $V_2$-invariant. Relative to the decomposition $\Hil=\mc{E}_\grl\oplus \mc{E}_\grl^\bot$, we write
  \[ V_1=\pvect{\grl & 0 \\ 0 & W}, \quad V_2=\pvect{A & B \\ 0 & C}. \]
  Since $\grl$ is an eigenvalue of an isometry, $|\grl|=1$. Note that $W$ does not have $\grl$ as an eigenvalue, and therefore $W^*$ does not have $\cc{\grl}$ as an eigenvalue. Commutativity implies $B(W-\grl)=0$ and so $(W^*-\cc{\grl})B^*=0$. Therefore $B^*=0$ and thus $V_2=A\oplus C$.
\end{proof}

\begin{lemma}\label{Decomp:EgnLem2}
  Let $\VV$ be an $n$-tuple of commuting isometries on a Hilbert space $\Hil$. There exists a subset $\mc{F}\sbse (\{0\}\cup\TT)^n$ and a collection of non-zero pairwise orthogonal $\VV$-reducing subspaces $(\Hil_z)_{z\in\mc{F}}$ such that $\Hil=\bigoplus_{z\in\mc{F}}\Hil_z$, and for every $z\in\mc{F}$ and $i\in\{1,\dots,n\}$
  \begin{enumerate}
    \item[\rm{(1)}] if $z_i=0$, then $V_i|\Hil_z$ has no eigenvalues; else
    \item[\rm{(2)}] if $z_i\neq 0$, then $(V_i-z_i I)|\Hil_z=0$.
  \end{enumerate}
\end{lemma}
\begin{proof}
  Given $i\in\{1,\dots,n\}$ and $\grl\in\TT$, we denote by $P_{i,\grl}$ the orthogonal projection onto $\ker(V_i-\grl I)$. Then $P_{i,\grl}P_{i,\eta}=0$ when $\grl$ and $\eta$ are distinct elements of $\TT$, and $V_iP_{i,\grl}=\grl P_{i,\grl}$. We define $P_{i,0}$ to be projection onto $\bigcap_\grl \ker(V_i-\grl I)^\bot$. Equivalently, $P_{i,0}=I-\sum_{\grl\in\grs_p(V_1)} P_{i,\grl}$, where the sum converges in the strong operator topology and $\grs_p(V_1)$ denotes the point spectrum of $V_1$.
  
  Lemma \ref{Decomp:EgnLem1}(1) implies that $V_i$ commutes with $P_{j,\grl}$ for every $i,j\in\{1,\dots,n\}$ and $\grl\in\TT$. We also note that $(V_i-\eta I)P_{j,\grl}P_{i,\eta}=0$ and therefore $P_{i,\eta}P_{j,\grl}P_{i,\eta}=P_{j,\grl}P_{i,\eta}$. Thus $P_{j,\grl}$ and $P_{i,\eta}$ commute, and so $P_{i,w_1}$ and $P_{j,w_2}$ commute for $w_1,w_2\in \{0\}\cup\TT$. Given $z\in (\{0\}\cup\TT)^n$, we set $P_z=P_{1,z_1}\cdots P_{n,z_n}$. Then $P_z$ is an orthogonal projection that commutes with $\VV$. Let $\mc{F}$ be the set of all $z\in(\{0\}\cup\TT)^n$ such that $P_z\neq 0$. Since $I=\sum_\grl P_{i,w}$ for each $i$, we have $I=\sum_{z\in\mc{F}}P_z$. If $z,z'\in \mc{F}$ are distinct then $P_z$ is orthogonal to $P_{z'}$, and if $z_i\neq 0$ then $V_iP_z=z_iP_z$.
  
  Finally, suppose $z\in\mc{F}$ and $i\in\{1,\dots,n\}$ are such that $z_i=0$, and suppose that $h\in \ran P_z$ so that $V_ih=\grl h$ for some $\grl\in\TT$. Then $P_{i,\grl}h=h$ and $P_zh=h$, and so $P_zP_{i,\grl}h=h$. However, $P_z=P_zP_{i,0}$ and $P_{i,0}P_{i,\grl}=0$. Thus $h=0$ and therefore $V_i|\ran P_z$ has no eigenvalues.
\end{proof}

Notice that if $\VV=\VV^{(1)}\oplus \VV^{(2)}$, then $Z(\Ann(\VV^{(1)}))\sbse Z(\Ann(\VV))$. Thus, when $Z(\Ann(\VV))$ has dimension 1, each summand of the preceding decomposition does likewise. Furthermore, when $\VV$ is completely non-unitary, each summand is also completely non-unitary. As such, we can restrict our attention to the following case.

\begin{lemma}\label{Decomp:OneAtATime}
  Let $\VV=(V_1,\dots,V_s,\grl_{s+1}I,\dots,\grl_n I)$ be a completely non-unitary $n$-tuple of isometries on a Hilbert space $\Hil$, where $\grl_{s+1},\dots,\grl_n\in\TT$. If $V_1,\dots,V_s$ have no eigenvalues and $Z(\Ann(\VV))$ has dimension $1$, then $V_1,\dots,V_s$ are shifts and there is a distinguished variety $\mc{W}$ in $\CC^s$ of dimension $1$ such that
  \[ Z(\Ann(\VV))=\mc{W}\times\{(\grl_{s+1},\dots,\grl_n)\}. \]
\end{lemma}
\begin{proof}
  If $\Ann(\VV)$ is prime, then the conclusion follows from Theorem \ref{PrmIdl:MainThm}. Suppose that $\Ann(\VV)$ has prime factors $\mc{I}_1,\dots,\mc{I}_m$ and $m>1$. For each $i$, set $\what{\mc{I}}_i=\bigcap_{j\neq i}\mc{I}_j$, and let $\Hil_i$ be the closure of $\what{\mc{I}}_i(\VV)\Hil$. Then $\Hil_i\neq \{0\}$ and $\mc{I}_i=\Ann(\VV|\Hil_i)$. By Theorem \ref{PrmIdl:MainThm} we have that $Z(\mc{I}_i)=\mc{W}_i\times\{(\grl_{s+1},\dots,\grl_n)\}$, where $\mc{W}_i$ is a distinguished variety. Thus
  \[ Z(\Ann(\VV))=\bigcup_{i=1}^m Z(\mc{I}_i)=\left(\bigcup_{i=1}^m\mc{W}_i\right)\times\{(\grl_{s+1},\dots,\grl_n)\}. \]
  For $i\in\{1,\dots,m\}$ and $\ell\in\{1,\dots,s-1\}$, we see in the proof of Theorem \ref{PrmIdl:MainThm} that there is an irreducible polynomial $p_\ell^{(i)}(X_\ell,X_s)\in \mc{I}_i$ which is non-constant in both variables. Setting $p_\ell(X_\ell,X_s)=\prod_{i=1}^m p_\ell^{(i)}(X_\ell,X_s)$, we see that $p_\ell(X_\ell,X_s)$ is in $\Ann(\VV)$ and has no single variable factors. By Corollary \ref{PrmIdl:NSVCor}, $V_1,\dots,V_s$ are shifts.
\end{proof}

We now summarize the results of this section in the following.

\begin{theorem}\label{Decomp:MainThm}
  Let $\VV$ be a completely non-unitary $n$-tuple of commuting isometries, and suppose $Z(\Ann(\VV))$ has dimension $1$. There is a subset $\mc{F}\sbse (\{0\}\cup\TT)^n$ and a collection of non-zero $\VV$-reducing subspaces $(\Hil_z)_{z\in\mc{F}}$ such that $\Hil=\bigoplus_{z\in\mc{F}}\Hil_z$, and for each $z\in\mc{F}$ the following hold.
  \begin{enumerate}
    \item[\rm{(1)}] If $i\in\{1,\dots,n\}$ and $z_i=0$, then $V_i|\Hil_z$ is a shift. If $z_i\neq 0$, then $(V_i-z_i I)|\Hil_z=0$.
    \item[\rm{(2)}] Suppose exactly $s$ components of $z$ are 0, and let $\grl_{s+1},\dots,\grl_n$ denote the non-zero components of $z$. After a permutation of coordinates, there is a distinguished variety $\mc{W}\sbse \CC^s$ of pure dimension $1$ so that
    \[ Z(\Ann(\VV|\Hil_z))=\mc{W}\times\{(\grl_{s+1},\dots,\grl_n)\}.  \]
  \end{enumerate}
\end{theorem}

\section{Finite Multiplicity}

In this section we suppose $\VV$ is an $n$-tuple of commuting shifts. When the dimension of $Z(\Ann(\VV))$ is 1, we show in Lemma \ref{FinalSec:MainThm} that $\VV$ has a finite cyclic set if and only if $V_i$ has finite multiplicity for some $i$. When each $V_i$ is known to have finite multiplicity, we show in Proposition \ref{FinMult:DimForFree} that the condition $\dim Z(\Ann(\VV))=1$ automatically satisfied.

In what follows, we set $\XX'=(X_1,\dots,X_{n-1})$.

\begin{lemma}\label{FinMult:DimZLem}
  Let $\VV$ be an $n$-tuple of commuting shifts and suppose that the dimension of $Z(\Ann(\VV))$ is $1$. Then $\{p(\XX',0):p\in\Ann(\VV)\}$ has finite linear codimension in $\CC[\XX']$. 
\end{lemma}
\begin{proof}
  Let $\mc{I}_1,\dots,\mc{I}_m$ be the prime factors of $\Ann(\VV)$. It follows from Theorem \ref{CharAnn:MainThm} that $X_n\nin \mc{I}_j$ for any $j$. This implies that $Z(\mc{I}_j)\cap\{z:z_n=0\}$ is a proper subvariety of $Z(\mc{I}_j)$ and thus of dimension 0. Since this holds for each $j$, we see that $Z(\Ann(\VV))$ meets $\{z:z_n=0\}$ only at finitely many points.
  
  We set $\mc{J}=\{p(\XX',0):p\in\Ann(\VV)\}$ and note that this is an ideal of $\CC[\XX']$ isomorphic to $(\Ann(\VV)+\langle X_n\rangle)/\langle X_n\rangle$, where $\langle X_n\rangle$ is the ideal generated by $X_n$ in $\CC[\XX]$. Because $\CC[\XX']/\mc{J}$ is isomorphic to $\CC[\XX]/(\Ann(\VV)+\langle X_n\rangle)$, it follows that $\mc{J}$ has finite linear codimension in $\CC[\XX']$.
\end{proof}

Recall that a shift $V$ has a finite cyclic set if and only if $V$ has finite multiplicity. In particular, any basis for $\ker V^*$ provides such a cyclic set. Thus, when $\VV$ is an $n$-tuple of shifts of finite multiplicity, it follows that any basis for $\ker V_n^*$ is a finite cyclic set for $\VV$. When $Z(\Ann(\VV))$ has dimension 1, we have the following.

\begin{theorem}\label{FinMult:Equiv}
    Let $\VV$ be an $n$-tuple of commuting shifts and suppose that the dimension of $Z(\Ann(\VV))$ is $1$. Then the following assertions are equivalent.
    \begin{enumerate}
        \item[\rm{(1)}] There is a finite cyclic set for $\VV$.
        \item[\rm{(2)}] $V_i$ has finite multiplicity for each $i$.
        \item[\rm{(3)}] $V_i$ has finite multiplicity for at least one $i$.
    \end{enumerate}
\end{theorem}
\begin{proof}
  Clearly (2)$\SUF$(3)$\SUF$(1). Suppose that (1) is true, and let $\{h_1,\dots,h_k\}$ be a cyclic set for $\VV$. We only show that $V_n$ has finite multiplicity, as the same argument applies to an arbitrary permutation of $V_1,\dots,V_n$. Fix an $f\in\ker V_n^*$ and note that
  \[ \sum_{i=1}^k\inner{Q_i(\VV)h_i}{f}=\sum_{i=1}^k\inner{Q_i(\VV',0)h_i}{f}, \quad Q_1,\dots,Q_k\in\CC[\XX] \]
  where $\VV'=(V_1,\dots,V_{n-1})$. By Lemma \ref{FinMult:DimZLem}, there are $r_1,\dots,r_\ell\in \CC[\XX']$ so that
  \[ \CC[\XX']=\CC\cdot r_1+\cdots+\CC\cdot r_\ell+\{p(\XX',0):p\in\Ann(\VV)\}. \]
  Given $Q_1,\dots,Q_k\in\CC[\XX]$, there are $a_{i1},\dots,a_{i\ell}\in\CC$ and $p_i\in\Ann(\VV)$ such that
  \[ Q_i(\XX',0)=p_i(\XX',0)+\sum_{j=1}^\ell a_{ij}r_j, \quad i=1,\dots,k. \]
  As $\inner{p_i(\VV',0)h_i}{f}=\inner{p_i(\VV)h_i}{f}=0$, we have
  \[ \sum_{i=1}^k\inner{Q_i(\VV)h_i}{f}=\sum_{j=1}^\ell\sum_{i=1}^ka_{ij}\inner{r_j(\VV')h_i}{f}. \] 
  Let $\mc{L}$ be the space spanned by $r_j(\VV')h_i$ as $j$ goes from 1 to $\ell$ and $i$ from 1 to $k$. If $f\bot \mc{L}$, then cyclicity implies that $f=0$, and thus we see that $\mc{L}^\bot\cap\ker V_n^*=\{0\}$. Since $\mc{L}$ has finite dimension, it follows that $\ker V_n^*$ does as well.
\end{proof}

Let $\VV$ be as above. Using the notation of Theorem \ref{CharAnn:MainThm} and its proof, recall that the spaces $\Hil_1,\dots,\Hil_m$ in the statement of Theorem \ref{CharAnn:MainThm} are equal to the closures of $\what{\mc{I}}_1(\VV)\Hil,\dots,\what{\mc{I}}_m(\VV)\Hil$, respectively. Since $\VV$ has a finite cyclic set $S$ and $\sum_{i=1}^m\what{\mc{I}}_i$ has finite linear codimension in $\CC[\XX]$, it follows that $\mc{K}=\Hil_1\oplus\dots\oplus\Hil_m$ has finite codimension in $\Hil$. Thus there is a finite codimensional $\VV$-invariant subspace $\mc{K}$ of $\Hil$ so that
    \[ \VV|\mc{K}=(\VV|\Hil_1)\oplus\dots\oplus (\VV|\Hil_m) \]
    with $\Ann(\VV|\Hil_i)=\mc{I}_i$ for $i=1,\dots,m$.

We also observe that when each $V_j$ has finite multiplicity, it is unnecessary to assume $\dim Z(\Ann(\VV))=1$, as we see in the following proposition. We remark, as before, that this result is essentially a corollary of work in \cite{AgMc2005}.

\begin{proposition}\label{FinMult:DimForFree}
    Let $\VV$ be an $n$-tuple of commuting shifts of finite multiplicity. Then $Z(\Ann(\VV))$ has dimension $1$.
\end{proposition}
\begin{proof}
    Let $k=\dim\ker V_n^*$, and represent $\VV$ as an $n$-tuple of Toeplitz operators $(T_{\grJ_1},\dots,T_{\grJ_{n-1}},T_{\grz\cdot I})$ on $H^2(\DD)\otimes \CC^k$, where $\grJ_1,\dots,\grJ_{n-1}$ are matrix valued inner functions. Since $\dim\ker V_i^*<\infty$ for $i=1,\dots,n-1$ as well, the matrix entries of $\grJ_i$ are rational functions. There are then polynomials $P_i\in\CC[X_i,X_n]$ and $Q_i\in\CC[X_n]$ so that $\det(w\cdot I-\grJ_i(z))=P_i(w,z)/Q_i(z)$. Let $\mc{I}$ denote the ideal generated by $P_1,\dots,P_{n-1}$, and note that $\mc{I}\sbse \Ann(\VV)$. As $\dim Z(\mc{I})\leq 1$, the proposition follows.
\end{proof}

\section{Near Unitary Equivalence}

Suppose $\VV$ is an $n$-tuple of commuting isometries on a Hilbert space $\Hil$ and $\WW$ an $n$-tuple of commuting isometries on a Hilbert space $\mc{K}$. If there exists a $\WW$-invariant subspace $\mc{K}'$ of finite codimension in $\mc{K}$ and a unitary operator $U$ from $\Hil$ onto $\mc{K}'$ so that $UV_i=W_iU$ for $i=1,\dots,n$, then we write $\VV\lesssim \WW$. If $\VV\lesssim \WW$ and $\WW\lesssim \VV$, we say that $\VV$ and $\WW$ are \textit{nearly unitarily equivalent} and write $\VV\approx \WW$.

\begin{example}
  Let $V_1$ and $V_2$ denote multiplication by $\pvect{z & 0\\0 & z}$ and $\pvect{0 & z\\ z & 0}$ on the $\CC^2$-valued Hardy space $H^2(\DD)\otimes\CC^2$, respectively. We denote by $\Hil_0$ the cyclic subspace generated by $(1,0)^t$. As $(V_1,V_2)$ is not cyclic, $(V_1,V_2)$ and $(V_1,V_2)|\Hil_0$ are not unitarily equivalent. However, $\Hil_0^\bot=\CC\cdot (0,1)^t$, and thus $(V_1,V_2)|\Hil_0\lesssim (V_1,V_2)$. In fact, it follows from Lemma \ref{NUE:PrecLem} that these two pairs are nearly unitarily equivalent.
\end{example}

The purpose of this section is to prove the following result.

\begin{theorem}\label{FinalSec:MainThm}
  Let $\VV$ and $\WW$ be two cyclic $n$-tuples of commuting shifts. If $\Ann(\VV)=\Ann(\WW)$ and $\dim Z(\Ann(\VV))=1$, then $\VV$ and $\WW$ are nearly unitarily equivalent.
\end{theorem}

\noindent The proof is divided into three parts, approximately following what is done in \cite{AKM}. First we show, by Lemma \ref{NUE:ApplyDecomp}, that it is sufficient to consider the case where $\Ann(\VV)$ is prime. We then construct an $n$-tuple of commuting shifts in Lemma \ref{NUE:HarmIsomLem} based on the desingularization of $Z(\Ann(\VV))\cap\cc{\DD^n}$ that serves as a common model for all $n$-tuples with annihilator equal to $\Ann(\VV)$. The proof is finally completed with Lemma \ref{NUE:KKLemma} by showing that $\VV$ is nearly unitarily equivalent to this model. Before this, however, we require a few additional facts.

\begin{lemma}\label{NUE:NotInRange}
  Let $\VV$ be a $n$-tuple of commuting shifts of finite multiplicity. For any $z_n\in \DD$, there exist $z_1,\dots,z_{n-1}\in\cc{\DD}$ so that $(\cc{z}_1,\dots,\cc{z}_n)$ is in the joint point spectrum of $\VV^*$.
\end{lemma}
\begin{proof}
  Since $V_n$ is a shift of finite multiplicity, and $V_1,\dots,V_{n-1}$ commute with $V_n$, we represent $\VV$ as an $n$-tuple of matrix valued Toeplitz operators \[\VV=(T_{\grJ_1},\dots,T_{\grJ_{n-1}},T_{\grz\cdot I_k})\] on $H^2(\DD)\otimes\CC^k$, where $k=\dim\ker V_n^*$, $I_k$ is the identity on $\CC^k$, $\grz$ is the coordinate function on $\DD$, and $\grJ_1,\dots,\grJ_{n-1}$ are matrix valued inner functions on $\DD$. Let $g$ be a non-zero vector in $H^2(\DD)$ so that $T_\grz^* g= \cc{z}_ng$. Then $T_{\grJ_i}^*(g\otimes u)=g\otimes \grJ_i(z_n)^*u$ for any $u\in\CC^k$. Since $\grJ_1(z_n)^*,\dots,\grJ_{n-1}(z_n)^*$ are commuting matrices, they have a common eigenvector $v$, say with $\grJ_1(z_n)^*v=\cc{z}_1v$ $,\dots,$ $\grJ_{n-1}(z_n)^*v=\cc{z}_{n-1}v$. Thus
  \[ T_{\grJ_i}^*(g\otimes v)=\cc{z}_i\cdot(g\otimes v) \]
  for $i=1,\dots,n-1$.
\end{proof}

We remark that when $\dim Z(\Ann(\VV))=1$, the preceding lemma holds without the assumption of finite multiplicity. This is due to the fact that each operator $\grJ_j(z)$ will, in this case, have a non-zero annihilating polynomial.

We also note that Lemma \ref{NUE:PolyMapLem}(2) is well-known, but we sketch the proof nevertheless.

\begin{lemma}\label{NUE:PolyMapLem}
  Let $\VV$ be an $n$-tuple of commuting isometries on a Hilbert space $\Hil$, and let $Q\in \CC[X_n]\bksl\{0\}$.
  \begin{enumerate}
    \item[\rm{(1)}] If $V_n$ has no eigenvalues, then $\VV|\cc{Q(V_n)\Hil}$ is unitarily equivalent to $\VV$.
    \item[\rm{(2)}] If $V_n$ is a shift of finite multiplicity, then $\cc{Q(V_n)\Hil}$ has finite codimension in $\Hil$.
  \end{enumerate}
\end{lemma}
\begin{proof}
  (1) Since $V_n$ has no eigenvalues, $V_n^*$ has no eigenvalues in $\TT$. Thus $\ran(V_n-\grl)$ is dense whenever $|\grl|\geq 1$. We therefore suppose that $Q(X_n)=\prod_{j=1}^k (X_n-\grl_k)$ where $\grl_1,\dots,\grl_k\in \DD$. The operator $B=\prod_{j=1}^k\frac{V_n-\grl_k}{1-\cc{\grl}_kV_n}$ is an isometry that commutes with $\VV$, and $B\Hil=\cc{\ran Q(V_n)}$.
  
  \vspace{10pt}
  \noindent (2) We show that $Q(V_n)^*$ has a finite dimensional kernel. For this, it suffices to note that if $\mc{K}$ is a finite dimensional subspace, then $\{f\in\Hil: (V_n-\grl)^*f\in\mc{K}\}$ is finite dimensional for any $\grl\in\CC$. 
\end{proof}

\begin{lemma}\label{NUE:PrecLem}
    Let $\VV$ and $\WW$ be $n$-tuples of commuting isometries. Suppose $\WW\lesssim \VV$ and at least one element of $\VV$ has no eigenvalues. Then :
    \begin{enumerate}
        \item[\rm{(1)}] $\Ann(\VV)=\Ann(\WW)$; and
        \item[\rm{(2)}] $\WW$ is an $n$-tuple of commuting shifts of finite multiplicity if and only if $\VV$ is an $n$-tuple of commuting shifts of finite multiplicity, in which case $\VV\approx \WW$.
    \end{enumerate}
\end{lemma}
\begin{proof}
  We may (and do) suppose that there is a $\VV$-invariant finite codimensional subspace $\mc{K}$ of $\Hil$ such that $\WW=\VV|\mc{K}$. By a permutation of coordinates, we also suppose that $V_n$ has no eigenvalues. Denote by $A$ the compression of $V_n$ to $\mc{K}^\bot$. Since $\mc{K}^\bot$ has finite dimension, there is a non-zero polynomial $Q\in\CC[X_n]$ so that $Q(A)=0$ and thus $Q(V_n)\Hil\sbse \mc{K}$. By Lemma \ref{NUE:PolyMapLem}, $\VV|\cc{Q(V_n)\Hil}$ is unitarily equivalent to $\VV$, and so
  \[ \Ann(\VV)\sbse \Ann(\VV|\mc{K})\sbse \Ann(\VV|\cc{Q(V_n)\Hil})=\Ann(\VV). \]
  
  For any $i\in\{1,\dots,n\}$,
  \[ \bigcap_{j=0}^\infty V_i^j \cc{Q(V_n)\Hil}\sbse \bigcap_{j=0}^\infty V_i^j\mc{K}\sbse \bigcap_{j=0}^\infty V_i^j\Hil.  \]
  As $V_i|\cc{Q(V_n)\Hil}$ is unitarily equivalent to $V_i$, it follows that $V_i$ is a shift if and only if $V_i|\mc{K}$ is a shift. We also have the following inclusions;
  \begin{equation}\label{NUE:Chain1}
    V_i\cc{Q(V_n)\mc{K}} \sbse V_i\mc{K} \sbse \mc{K} \sbse \Hil
  \end{equation}
  and
  \begin{equation}\label{NUE:Chain2}
    V_i\cc{Q(V_n)\mc{K}} \sbse V_i\cc{Q(V_n)\Hil} \sbse \cc{Q(V_n)\Hil} \sbse \Hil.
  \end{equation}
  As $\mc{K}$ has finite codimension in $\Hil$, we know that $V_i\cc{Q(V_n)\mc{K}}$ has finite codimension in $V_i\cc{Q(V_n)\Hil}$. When each $V_j$ has finite multiplicity, it follows that every subspace in \eqref{NUE:Chain2} has finite codimension in the next. Thus $V_i\mc{K}$ has finite codimension in $\mc{K}$. When each $V_j|\mc{K}$ has finite multiplicity, it follows that every subspace in \eqref{NUE:Chain1} has finite codimension in the next. Then $V_i|\cc{Q(V_n)\Hil}$ has finite multiplicity. Because $V_i|\cc{Q(V_n)\Hil}$ is unitarily equivalent to $V_i$, it follows that $V_i$ has finite multiplicity.
  
  Finally, we note that $Q(V_n)\mc{K}\sbse Q(V_n)\Hil\sbse \mc{K}\sbse \Hil$. Thus, if either $\VV$ or $\VV|\mc{K}$ is an $n$-tuple of shifts of finite multiplicity, then $\cc{Q(V_n)\Hil}$ has finite codimension in $\mc{K}$. In other words,
  \[ \VV|\cc{Q(V_n)\Hil}\lesssim \VV|\mc{K}\lesssim \VV. \]
  Because $\VV|\cc{Q(V_n)\Hil}$ is unitarily equivalent to $\VV$, we have $\VV\approx \VV|\mc{K}$.
\end{proof}

\begin{corollary}\label{NUE:JCor}
  Let $\VV$ be an $n$-tuple of commuting shifts of finite multiplicity on a Hilbert space $\Hil$ such that $Z(\Ann(\VV))$ has dimension $1$. Denote by $\mc{I}_1,\dots,\mc{I}_m$ the prime factors of $\Ann(\VV)$ and set $\mc{J}=\sum_{i=1}^m\bigcap_{j\neq i}\mc{I}_j$. Then $\VV|\cc{\mc{J}(\VV)\Hil}$ is nearly unitarily equivalent to $\VV$.
\end{corollary}
\begin{proof}
  By Corollary \ref{Intro:FctrsLem}, there is a finite linear dimensional subspace $L$ of $\CC[\XX]$ so that $\CC[\XX]=L+\mc{J}$. Since each $V_i$ has finite multiplicity, there is a finite cyclic set $C$ in $\Hil$ for $\VV$. Thus $L(\VV) C+\mc{J}(\VV)\Hil$ is dense in $\Hil$, and so $\cc{\mc{J}(\VV)\Hil}$ has finite codimension in $\Hil$. The corollary now follows from Lemma \ref{NUE:PrecLem}.
\end{proof}

Let $\mc{I}$ be a radical ideal of $\CC[\XX]$ with $\dim Z(\mc{I})=1$, and let $\mc{I}_1,\dots,\mc{I}_m$ the prime factors of $\mc{I}$. We set $\what{\mc{I}}_i=\bigcap_{j\neq i}\mc{I}_j$ for each $i$ and $\mc{J}=\sum_{i=1}^m\what{\mc{I}}_i$. Suppose there are two $n$-tuples of commuting shifts of finite multiplicity, $\VV$ on $\Hil$ and $\WW$ on $\mc{K}$, so that $\Ann(\VV)=\Ann(\WW)=\mc{I}$. We define $\Hil_i$ to be the closure of $\what{\mc{I}}_i(\VV)\Hil$ and $\mc{K}_i$ to be the closure of $\what{\mc{I}}_i(\WW)\mc{K}$. Theorem \ref{CharAnn:MainThm} then states that
\[ \VV|\cc{\mc{J}(\VV)\Hil}=\bigoplus_{i=1}^m \VV|\Hil_i, \quad \WW|\cc{\mc{J}(\WW)\mc{K}}=\bigoplus_{i=1}^m \WW|\mc{K}_i. \]
Thus, to show that $\VV$ and $\WW$ are nearly unitarily equivalent, it suffices to prove that $\VV|\Hil_i$ is nearly unitarily equivalent to $\WW|\mc{K}_i$ for $i=1,\dots,n$. However, we note that even if $\VV$ and $\WW$ are both cyclic, it may be that neither $\VV|\Hil_i$ nor $\WW|\mc{K}_i$ is cyclic. The purpose of Lemma \ref{NUE:ApplyDecomp} is to address this issue. First, however, we recall a result from the theory of subnormal operators, for which we need the following notation. Given a finite positive Borel measure $\mu$ on $\TT^n$, we denote by $P^2(\mu)$ the $L^2$-closure of $\CC[\XX]$, where here for $i=1,\dots,n$ we identify $X_i$ with the coordinate function $z\mapsto z_i$. If $f$ is a bounded Borel function on $\TT^n$, we denote by $M_{f,\mu}$ the operator of multiplication by $f$ acting on $P^2(\mu)$, and set $M_{\XX,\mu}=(M_{X_1,\mu},\dots,M_{X_n,\mu})$.

\begin{lemma}\label{NUE:GetMu}
  Let $\VV$ be a cyclic $n$-tuple of commuting shifts on a Hilbert space $\Hil$. There is a diffuse finite positive Borel measure $\mu$ concentrated in $Z(\Ann(\VV))\cap\TT^n$ so that $\VV$ is unitarily equivalent to $M_{\XX,\mu}$. 
\end{lemma}
\begin{proof}
  Denote by $v_0$ a cyclic vector for $\VV$. We set $\mc{V}=Z(\Ann(\VV))$, and denote by $\wtil{\VV}$ the minimal unitary extension of $\VV$ and $\grs(\wtil{\VV})$ the Taylor joint spectrum of $\VV$. By \cite[Thm. IV.7.26]{Vasilescu}, there is a projection valued measure $E$ supported on $\grs(\wtil{\VV})$ such that $\wtil{V}_i=\int_{\grs(\wtil{\VV})} z_i dE(z)$ for each $i$. As $\wtil{\VV}|\Hil=\VV$, for any $p,q\in\CC[\XX]$
  \[ \inner{p(\VV)v_0}{q(\VV)v_0}=\int_{\grs(\wtil{\VV})}p(z)\cc{q(z)}d\mu(z) \]
  where $\mu(\cdot)=\inner{E(\cdot)v_0}{v_0}$. It follows from \cite[Thm. III.10.4]{Vasilescu} that $\grs(\wtil{\VV})\sbse \mc{V}\cap\TT^n$.
  
  Suppose $f\in\wtil{\Hil}$ and $z\in\TT^n$ so that $E(\{z\})f=f$. For any $g\in \Hil$ and $\gra,\grb\in\NN_0^n$ we have
  \[ |\inner{f}{\wtil{\VV}^{*\gra}g}|=|\inner{f}{g}|=|\inner{P_\Hil f}{\VV^\grb g}|. \]
  Here $P_\Hil$ is projection onto $\Hil$. Because $V_1^k g$ tends weakly to $0$ as $k\to\infty$, it follows that $\inner{f}{\wtil{\VV}^{*\gra}g}=0$ for any $g\in\Hil$ and $\gra\in\NN_0^n$. Since $\{\wtil{\VV}^{*\gra}g:g\in\Hil,\gra\in\NN_0^n\}$ is dense in $\wtil{\Hil}$, we have $f=0$. Thus $E(\{z\})=0$ and therefore $\mu$ has no atoms. 
\end{proof}

\begin{lemma}\label{NUE:ApplyDecomp}
  Let $\VV$ be a cyclic $n$-tuple of commuting shifts on a Hilbert space $\Hil$, and suppose $\dim Z(\Ann(\VV))=1$. We set $\mc{V}=Z(\Ann(\VV))$, and take $\mu$ to the measure provided by the preceding lemma. Denote by $\mc{V}_1,\dots,\mc{V}_m$ the irreducible components of $\mc{V}$ and $\mc{I}_1,\dots,\mc{I}_m$ the corresponding prime ideals. Suppose $m>1$, and let $\nu$ be another finite diffuse positive Borel measure on $\mc{V}\cap\TT^n$. For each $i\in\{1,\dots,m\}$ we write $\mu_i$ and $\nu_i$ for the restriction of $\mu$ and $\nu$ to $\mc{V}_i\cap\TT^n$, respectively. Then :
  \begin{enumerate}
    \item[\rm{(1)}] $M_{\XX,\mu_i}$ is an $n$-tuple of shifts of finite multiplicity with $\Ann(M_{\XX,\mu_i})=\mc{I}_i$ for $i=1,\dots,n$; and
    \item[\rm{(2)}] if $M_{\XX,\mu_i}\lesssim M_{\XX,\nu_i}$ for each $i$, then $M_{\XX,\mu}\approx M_{\XX,\nu}$.
  \end{enumerate}
\end{lemma}
\begin{proof}
    Write $\what{\mc{I}}_i=\bigcap_{j\neq i}\mc{I}_j$ for each $i$, and note that if $p\in \what{\mc{I}}_i$ then $p(\mc{V}_j)=0$ whenever $j\neq i$. For any $p,q\in\what{\mc{I}}_i$ we have
  \[ \int p \cc{q}d\mu=\int p\cc{q}d\mu_i. \]
  Thus $M_{\XX,\mu}$ restricted to the $L^2(\mu)$-closure of $\what{\mc{I}}_i$ is equal to $M_{\XX,\mu_i}$ restricted to the $L^2(\mu_i)$-closure $\Hil_i$ of $\what{\mc{I}}_i$. Likewise, $M_{\XX,\nu}$ restricted to the $L^2(\nu)$-closure of $\what{\mc{I}}_i$ is equal to $M_{\XX,\nu_i}$ restricted to the $L^2(\nu_i)$-closure $\mc{K}_i$ of $\what{\mc{I}}_i$.
  
  \vspace{10pt}
  
  \noindent (1) Since the argument is essentially identical for the other components of $\mc{V}$, we consider only $i=1$. Let $p\in \what{\mc{I}}_1\bksl \mc{I}_1$ and note that $p$ has only finitely many zeros on $\mc{V}_1$. Because $\mu_1$ has no atoms, $M_{p,\mu_1}$ is injective. Thus, as $M_{\XX,\mu}|\cc{\ran M_{p,\mu}}=M_{\XX,\mu_1}|\cc{\ran M_{p,\mu_1}}$ is completely non-unitary, it follows that $M_{\XX,\mu_1}$ is completely non-unitary. Since $\Ann(M_{\XX,\mu_1})\spse\mc{I}_1$, Lemma \ref{CharAnn:Prime} implies that $\Ann(M_{\XX,\mu_1})=\mc{I}_1$, and so by Theorem \ref{PrmIdl:MainThm} we see that $M_{\XX,\mu_1}$ is an $n$-tuple of shifts. As $M_{\XX,\mu_1}$ is also cyclic, Theorem \ref{FinMult:Equiv} dictates that each $M_{X_j,\mu_1}$ has finite multiplicity.
  
  \vspace{10pt}
  
  \noindent (2) By the comments that follow Corollary \ref{NUE:JCor}, it suffices to prove that $M_{\XX,\mu_i}|\Hil_i$ and $M_{\XX,\nu_i}|\mc{K}_i$ are nearly unitarily equivalent for each $i$. Again, we only consider $i=1$.
  
  Since $M_{\XX,\mu_1}\lesssim M_{\XX,\nu_1}$, there is a unitary operator $U$ from $P^2(\mu_1)$ onto a finite codimensional subspace of $P^2(\nu_1)$ so that $UM_{X_j,\mu_1}=M_{X_j,\nu_1}U$ for each $j\in\{1,\dots,n\}$. Let $g_1,\dots,g_N$ be a basis for $P^2(\nu_1)\ominus UP^2(\mu_1)$, and let $p_1,\dots,p_k$ generate $\what{\mc{I}}_1$. We note that $\Hil_1$ is the closure of $\sum_{j=1}^k p_j(M_{\XX,\mu_1})P^2(\mu_1)$, and $\mc{K}_1$ is the closure of $\sum_{j=1}^k p_j(M_{\XX,\nu_1})P^2(\nu_1)$. Thus, as
  \[ \sum_{j=1}^k p_j(M_{\XX,\nu_1})P^2(\nu_1)=\sum_{j=1}^k\sum_{\ell=1}^N \CC\cdot p_j(M_{\XX,\nu_1})g_\ell+U\left(\sum_{j=1}^k p_j(M_{\XX,\mu_1})P^2(\mu_1)\right), \]
  it follows that $U\Hil_1$ has finite codimension in $\mc{K}_1$, and so
  \[ M_{\XX,\mu_1}|\mc{H}_1\lesssim M_{\XX,\nu_1}|\mc{K}_1. \]
  Because $M_{\XX,\mu_1}|\mc{H}_1$ is an $n$-tuple of shifts with annihilator $\mc{I}_1$ and a finite cyclic set, it follows from Theorem \ref{FinMult:Equiv} that $M_{X_i,\mu_1}|\mc{H}_1$ has finite multiplicity for each $i$. Since $\nu_1$ has no atoms, it follows that no element of $M_{\XX,\nu_1}|\mc{K}_1$ is multiplication by a scalar, and thus $M_{\XX,\mu_1}|\Hil_1\approx M_{\XX,\nu_1}|\mc{K}_1$ by Lemma \ref{NUE:PrecLem}.
  \end{proof} 
  
  Thus we restrict ourselves to the case where $\Ann(\VV)$ is prime and $\VV$ is of the form $M_{\XX,\mu}$, for $\mu$ a diffuse finite positive Borel measure on $Z(\Ann(\VV))\cap\TT^n$. In what follows, we use the following properties of $\mc{V}=Z(\Ann(\VV))$.
  \begin{enumerate}
    \item[\rm{(1)}] $\mc{V}$ is a distinguished variety; and
    \item[\rm{(2)}] if $z\in \mc{V}$ and $z_i\neq 0$ for each $i$, then $(1/\cc{z}_1,\dots,1/\cc{z}_n)\in \mc{V}$.
  \end{enumerate}
  Property (1) follows from Theorem \ref{PrmIdl:MainThm}, while property (2) follows from comments in section \ref{Sec:CharAnn}.
  
  
  Before moving onto the desingularization process, we make the following observation about the set $\mc{V}\cap\TT^n$, where we denote by $\mc{V}^*$ the regular set of $\mc{V}$.
  
\begin{lemma}\label{NUE:VcapT}
  $\mc{V}\cap\TT^n=\cc{\mc{V}\cap\DD^n}\bksl(\mc{V}\cap\DD^n)$, and for each point $y\in\mc{V}^*\cap\TT^n$ there is a neighborhood $U$ of $y$ in $\CC^n$ such that $\mc{V}\cap\TT^n\cap U$ is a simple smooth curve.
\end{lemma}
\begin{proof}
  Plainly $\mc{V}\cap\cc{\DD^n}\spse \cc{\mc{V}\cap\DD^n}$. Since $\mc{V}$ is distinguished, it follows that $\mc{V}\cap\TT^n\spse \cc{\mc{V}\cap\DD^n}\bksl(\mc{V}\cap\DD^n)$ and thus it suffices to show that $\mc{V}\cap\TT^n$ is contained in $\cc{\mc{V}\cap\DD^n}$. Let $z\in\mc{V}\cap\TT^n$, and first suppose that $z$ is a regular point of $\mc{V}$. By the implicit function theorem it follows, possibly after a permutation of coordinates, that there exists a disc $\grD_n$ about $z_n$ in $\CC$, a polydisc $\grD'$ about $(z_1,\dots,z_{n-1})$ in $\CC^{n-1}$, and a holomorphic function $\phi:\grD_n\to\grD'$ so that
  \[ \mc{V}\cap(\grD'\times\grD_n)=\{(\phi(w),w):w\in\grD_n\}. \]
  Let $w_1,w_2,\dots \in \grD_n\cap\DD$ so that $w_i\to z_n$. As $\mc{V}$ is distinguished, $(\phi(w_i),w_i)\in\mc{V}\cap\DD^n$ for each $i$ and thus $z$ is a limit point of $\mc{V}\cap\DD^n$. We also observe that the preceding parametrization shows $\mc{V}\cap\TT^n\cap(\grD'\times\grD_n)$ to be of the form $\{(\phi(e^{it}),e^{it}):a<t<b\}$ for some $a,b\in\RR$. That is, $\mc{V}\cap\TT^n\cap(\grD'\times\grD_n)$ is a simple smooth curve through $z$.
  
  Now suppose that $z$ is a singular point and $z^{(1)},z^{(2)},\dots\in\mc{V}^*$ are such that $z^{(i)}\to z$. Since $w\in\mc{V}$ implies $(1/\cc{w}_1,\dots,1/\cc{w}_n)\in \mc{V}$, we assume that $z^{(i)}\in\cc{\DD^n}$ for each $i$. Since each point of $\mc{V}^*\cap\TT^n$ is a limit point of $\mc{V}\cap\DD^n$, whenever $z^{(i)}\in \TT^n$ we can replace $z^{(i)}$ with a point in $\DD^n$ that is within a distance $1/i$ of the original point. Thus we produce a sequence in $\mc{V}\cap\DD^n$ converging to $z$.
\end{proof}

We remark that the last part of the proof can also be proved using the parametrization of $\mc{V}$ given by the desingularization discussed below without using the fact that $z\in \mc{V}\cap (\CC\bksl\{0\})^n$ implies $1/\cc{z}\in\mc{V}$.
  
  \subsection{Desingularizing $\mc{V}\cap\cc{\DD^n}$}\label{DesingSubsec}

  Our next objective is to construct a Riemann surface $\mc{S}$ with boundary, and a continuous map $h$ from $\cc{\mc{S}}$ onto $\mc{V}\cap\cc{\DD^n}$ so that $\pd\mc{S}=h^{-1}(\mc{V}\cap\TT^n)$ and $h$ is injective over $\mc{V}^*\cap\cc{\DD^n}$. Suppose we have such a map and let $\mu$ be a diffuse finite positive Borel measure on $\mc{V}\cap\TT^n$, as in Lemma \ref{NUE:GetMu}. Since $h$ is injective on the complement of a finite set, the pullback measure $\nu=\mu\of h$ is also a diffuse finite positive Borel measure and
  \[ \int_{\mc{V}\cap\TT^n}p\cc{q}d\mu=\int_{\pd\mc{S}} (p\of h)\cc{(q\of h)}d\nu, \quad p,q\in\CC[\XX]. \]
  Given a finite positive Borel measure $\tau$ on $\pd\mc{S}$, we denote by $A^2_h(\tau)$ the $L^2(\tau)$-closure of $\CC[h_1,\dots,h_n]$. If $f$ is a bounded Borel function on $\pd\mc{S}$, then we write $M_{f,\tau}$ for multiplication by $f$ on $A^2_h(\tau)$. Thus the preceding comments show that $M_{\XX,\mu}$ is unitarily equivalent to $M_{h,\nu}=(M_{h_1,\nu},\dots,M_{h_n,\nu})$.
  
  After constructing $\mc{S}$ and $h$, we find a diffuse finite positive Borel measure $\grw$ on $\pd\mc{S}$ so that $M_{h,\grw}$ is an $n$-tuple of shifts with annihilator $\mc{I}$. We then show that whenever $\nu$ is as above, then $M_{h,\nu}\approx M_{h,\grw}$. Once this done in Lemma \ref{NUE:KKLemma}, Theorem \ref{FinalSec:MainThm} is proved.
  
  The construction of $\mc{S}$ requires two steps. First we desingularize $\mc{V}$ to produce a Riemann surface $\mc{M}$ and a holomorphic map $H$ from $\mc{M}$ onto $\mc{V}$. The set $\mc{R}=H^{-1}(\mc{V}\cap\DD^n)$ may not, as a subset of $\mc{M}$, have the structure of a Riemann surface with boundary. Thus in the second step we embed $\mc{R}$ into another Riemann surface where $\mc{R}$ has the appropriate structure; we call this set $\mc{S}$.

We note that desingularization often requires a rotation of the variety before the surface can be constructed. In our case, however, we wish to use the fact that $\mc{V}$ is a distinguished variety, and a generic rotation may not preserve this property of $\mc{V}$. We therefore carry out the desingularization $\mc{V}$ without such rotations, and thus require some material from the theory of analytic sets. For this, we follow Chapter I of \cite{Chirka}. We say that $A\sbse \CC^n$ is an \textit{analytic subset of $\CC^n$} if for each point $x\in A$ there is a neighborhood $U$ about $x$ and functions $f_1,\dots,f_N$ holomorphic on $U$ such that $A\cap U=\{z\in U:f_1(z)=\dots=f_N(z)=0\}$. An analytic set $A$ is \textit{irreducible} if it cannot be written as the union of two analytic subsets distinct from $A$. Every analytic set is the union of an at most countable number of irreducible analytic sets, which are called the \textit{irreducible components} of $A$. Near any point $x$ in an analytic set $A$, there is a neighborhood $U'$ of $x$ so that $A\cap U'$ can be written as a finite union of irreducible analytic sets.  We note that any algebraic variety is an analytic set, but the restriction of an irreducible algebraic variety to an open subset may not be an irreducible analytic set. As an example, consider the complex curve $X_2^2=X_1^2(1+X_1)$ near the origin.

Recall that a map $f$ from a topological space $X$ into a topological space $Y$ is \textit{proper} if $f^{-1}(K)$ is compact whenever $K$ is a compact subset of $Y$. In addition to the preceding, we use the following results. The first of these follows from \cite[\S 1.3.1]{Chirka}, while the second is a special case of results in \cite[\S 1.3.2]{Chirka}.

\begin{proposition}\label{NUE:ChirkaPropLem}
  Assume that $X$ and $Y$ are Hausdorff, locally compact spaces, and $D\sbs X$, $G\sbs Y$ with $\cc{G}$ compact. If $A$ is a relatively closed subset of $D\times G$ that does not have limit points in $D\times \pd G$, then the restriction of the projection $(x,y)\mapsto x$ to $A$ is a proper map.
\end{proposition}

\begin{proposition}\label{NUE:ChirkaMapLem}
    Let $G=G'\times G''$, where $G'\sbs\CC^p,G''\sbs \CC^m$ are open, set $n=p+m$, and let $\pi:(z',z'')\mapsto z'$. If $A$ is an analytic subset of $G$ such that $\pi|A$ is a proper map, then $\pi(A)$ is an analytic subset of $G'$. Furthermore, if $\pi$ is also finite, then $\dim A=\dim \pi(A)$.
\end{proposition}

Let $U=U_1\times U'$ be a polydisc in $\CC^n$ centered at $0$ with $U_1\sbse \CC$ and $U'\sbse \CC^{n-1}$. Suppose that $A$ is an irreducible 1-dimensional analytic subset of $U$ and 0 is the only singular point of $A$. If $\pi:A\cap U\to U_1$ is a proper projection onto $U_1$, it follows from \cite[\S 1.3.7]{Chirka} that there is an integer $k>0$ so that $\pi|A\cap U\bksl\{0\}$ is a locally biholomorphic $k$-sheeted cover of $U_1\bksl\{0\}$. Starting from this, it is shown in \cite[\S 1.6]{Chirka} (see also \cite[\S 1.2]{Kollar}) that there is a number $\eta>0$ and a holomorphic homeomorphism $\grs$ from $\DD$ onto $S\cap U$ such that $\pi(\grs(w))=\eta w^k$ for each $w\in\DD$ while $\grs$ sends $\DD\bksl\{0\}$ biholomorphically onto $S\cap U\bksl\{0\}$. In order to produce the appropriate projection, we prove the following.

\begin{lemma}\label{NUE:ProjLem}
  Suppose $0\in\mc{V}$, and let $U$ be a neighborhood of $0$. Denote by $\pi$ the projection $(z_1,\dots,z_n)\mapsto z_1$. If $S$ is an irreducible analytic component of $\mc{V}\cap U$ containing $0$, then there is a polydisc $\grD$ centered at $0$ so that $\pi|(S\cap\grD)$ is a proper projection onto the disc $\pi(\grD)$.
\end{lemma}
\begin{proof}
  Let $L=\{0_1\}\times \CC^{n-1}$ where here $0_1$ denotes the origin in $\CC$. Since $\mc{V}$ is a distinguished variety, it follows that $\mc{V}\cap L$ is a compact variety and thus a finite set. By shrinking $U$, we arrange that $\mc{V}\cap L\cap U$ contains only the point 0. In particular, we have $S\cap L=\{0\}$.  Since $S$ is a relatively closed subset of $U$, there is a polydisc $D=D_1\times D'$ centered at 0, with $D'\sbse \CC^{n-1}$ and $D_1\sbse \CC$, so that $D\sbse U$ and $S$ is bounded away from $ \cc{D_1}\times \pd D'$. By Proposition \ref{NUE:ChirkaPropLem}, $\pi|(S\cap D)$ is a proper map into $D_1$.
  
  The analytic set $S\cap D$ has dimension 1, and so it follows from Proposition \ref{NUE:ChirkaMapLem} that $\pi(S\cap D)$ has dimension 1. But the only analytic subsets of $\CC$ with dimension 1 are open subsets. In particular, there is an open disc $\grD_1$ centered at $0_1$ in $\CC$ that is contained in $\pi(S\cap D)$. We set $\grD=\grD_1\times D'$, and claim that $\pi|S\cap\grD$ is also proper. Indeed, if $K$ is a compact subset of $\pi(S\cap\grD)=\grD_1$, then $K$ is also a compact subset of $\pi(S\cap D)$ and thus $(\pi|S\cap D)^{-1}(K)$ is compact. However, $(\pi|S\cap D)^{-1}(K)$ is subset of $S\cap \grD$, and thus $(\pi|S\cap\grD)^{-1}(K)=(\pi|S\cap D)^{-1}(K)$ is compact.
\end{proof}

We can now produce the desingularization $H:\mc{M}\to\mc{V}$. Let $\grS$ denote the set of singular points of $\mc{V}$. For each $z\in \grS$, we fix a neighborhood $U_z$ of $z$ such that $U_z\cap\mc{V}$ has finitely many irreducible analytic components, any two of which meet only at $z$, and $U_z\cap\grS=\{z\}$. For each irreducible analytic component $S$ of $U_z\cap\mc{V}$, let $D_{z,S}$ be a copy of $\DD$ and denote by $x_{z,S}$ the center of $D_{z,S}$. By Lemma \ref{NUE:ProjLem} and the comments preceding it, there is a polydisc $\grD_{z,S}$ centered at $z$ and a holomorphic homeomorphism $\grs_{z,S}$ from $D_{z,S}$ onto $S\cap\grD_{z,S}$ such that $\grs_{z,S}$ sends $D_{z,S}\bksl\{x_{z,S}\}$ biholomorphically onto $S\cap\grD_{z,S}\bksl\{z\}$. Furthermore, there is an integer $k>0$ and a number $\eta>0$ so that $\pi(\grs_{z,S}(w))=z_n+\eta w^k$ for each $w\in \DD$, where $\pi$ is projection onto the $n$-th coordinate. Let $X$ be the disjoint union of $\mc{V}^*$ with all of the $D_{z,S}$, and define $G:X\to \mc{V}$ by setting $G(x)=x$ when $x\in\mc{V}^*$ and $G(x)=\grs_{z,S}(x)$ if $x\in D_{z,S}$. Finally, we define $\mc{M}$ to be the Riemann surface that results from identifying $x,x'\in X$ when $G(x)=G(x')$ and denote by $H$ the resulting map from $\mc{M}$ onto $\mc{V}$. Note that $H$ is a proper finite holomorphic map and $H$ is biholomorphic over $\mc{V}^*$.

Our next objective is to describe the set $\mc{R}=H^{-1}(\mc{V}\cap\DD^n)$.

\begin{lemma}\label{NUE:Connect}
  $\mc{R}$ is connected.
\end{lemma}
\begin{proof}
  Since $\mc{M}$ consists of discs glued to $\mc{V}^*$, it suffices to show that $\mc{V}^*\cap\DD^n$ is path-connected. Fix $x,y\in \mc{V}^*$. Because $\mc{V}$ is irreducible, $\mc{V}^*$ is path-connected and so there is a path $\grg:[0,1]\to\mc{V}^*$ connecting $x$ and $y$. Then
  \[ \grb(t)=\begin{cases}
               \grg(t), & \grg(t)\in\mc{V}\cap\cc{\DD^n} \\
               1/\cc{\grg(t)}, & \grg(t)\in\mc{V}\cap\EE^n
             \end{cases}
 \]
  defines a path in $\mc{V}\cap\cc{\DD^n}$ connecting $x$ and $y$. Here $1/\cc{z}=(1/\cc{z}_1,\dots,1/\cc{z}_n)$ when $z=(z_1,\dots,z_n)\in(\CC\bksl\{0\})^n$. Since the singular set of $\mc{V}$ is finite, we modify $\grg$ so that $\grb$ is contained in $\mc{V}^*\cap\cc{\DD^n}$.  The path $\grb$ meets $\mc{V}\cap\TT^n$ only at regular points of $\mc{V}$, and so, by making the obvious modifications to $\grb$ near these points, we produce a path joining $x$ and $y$ entirely contained in $\mc{V}^*\cap\DD^n$.
\end{proof}

\begin{lemma}\label{NUE:BndryR}
  $\pd\mc{R}=H^{-1}(\mc{V}\cap\TT^n)$
\end{lemma}
\begin{proof}
  We immediately see that $\cc{\mc{R}}\sbse H^{-1}(\mc{V}\cap\cc{\DD^n})$ and so, by Lemma \ref{NUE:VcapT},
\[ \pd\mc{R}\sbse H^{-1}(\mc{V}\cap\TT^n). \]
Thus it suffices to show that $H^{-1}(\mc{V}\cap\TT^n)\sbse \cc{\mc{R}}$. Since $H$ is essentially the identity map over $\mc{V}^*$, we restrict our attention to singular points. Suppose $x\in H^{-1}(\mc{V}\cap\TT^n)$ so that $z=H(x)$ is singular, and let $x$ correspond to the center of $D_{z,S}$ for an irreducible analytic component $S$ of $\mc{V}\cap U_z$. For some neighborhood $W$ of $x$, the map $H|W$ is a homeomorphism of $W$ onto $S\cap\grD_{z,S}$. We recall that $\pi:(w_1,\dots,w_n)\mapsto w_n$ sends $S\cap\grD_{z,S}$ onto $\pi(\grD_{z,S})$, and that there is a number $\eta>0$ and an integer $k>0$ so that $\pi(\grs_{z,S}(w))=z_n+\eta w^k$ for every $w\in D_{z,S}$. Thus we can find a sequence $z^{(1)},z^{(2)},\dots \in S\cap\grD_{z,S}\cap\DD^n$ so that $z^{(j)}\to z$. Since $H|W$ is a homeomorphism onto $S\cap\grD_{z,S}$, the sequence $\{(H|W)^{-1}(z^{(i)})\}_i$ in $\mc{R}$ converges to $x$.
\end{proof}

We can say more about the boundary of $\mc{R}$. If $x\in\pd\mc{R}$ is such that $H(x)\in\mc{V}^*$, then we easily see that there is a neighborhood $W$ of $x$ so that $\mc{R}\cap W$ is diffeomorphic to $\{w\in\DD:\Im(z)>0\}$, with $\pd\mc{R}\cap W$ being sent smoothly onto the interval $(-1,1)$. Now suppose $x\in\pd\mc{R}$ so that $z=H(x)$ is a singular point of $\mc{V}$, and let $\pi:(w_1,\dots,w_n)\mapsto w_n$ on $\CC^n$. There exists a neighborhood $W$ of $x$, a biholomorphic map $\grs:\DD\to W$ with $\grs(0)=x$, a positive integer $k$, and a number $\eta>0$ so that $\pi(H(\grs(w)))=z_n+\eta w^k$ and $\pi$ sends $H(W)$ onto the disc $\grD_n$ of radius $\eta$ about $z_n$. In particular, $\pi(H(\mc{R}\cap W))=\DD\cap\grD_n$ and
\[ (\pi\of H\of \grs)^{-1}(\DD\cap\grD_n)=\{w\in\DD:|z_n+\eta w^k|<1\}. \]
Thus $\mc{R}\cap W$ is conformally equivalent to this subset of the plane. For $W$ and thus $\eta$ sufficiently small, we find that $\mc{R}\cap W$ is the union of $k$ disjoint simply connected open sets $W_1,\dots,W_k$. For each $j=1,\dots,k$, the set $W\cap\pd W_j$ consists of two simple smooth curves meeting only at $x$, and $\cc{W}_i\cap \cc{W}_j=\{x\}$ whenever $i\neq j$. When $k>1$, we call such an $x$ a \textit{star point} of $\mc{R}$. Note that the set of star points of $\mc{R}$ is a finite subset of $\pd\mc{R}$.

In proving Theorem \ref{FinalSec:MainThm}, we use some results concerning finite Riemann surfaces. A \textit{finite Riemann surface} is a domain $\grW$ in a Riemann surface with the property that $\cc{\grW}$ is compact and $\pd\grW$ consists of a finite number of disjoint simple closed smooth curves. Due to the possible existence of star points, our domain $\mc{R}$ may not be a finite Riemann surface as a subset of $\mc{M}$. We can, however, embed $\mc{R}$ into another Riemann surface where the star points are absent. This is done as follows. Let $B$ denote the set of star points of $\mc{R}$, and for a given $b\in B$ let $W$ be a coordinate neighborhood of $b$ in $\mc{M}$ so that $W\cap B=\{b\}$ and $W\cap\mc{R}=W_1\cup\dots\cup W_k$, as above. Glue $\cc{W}_i$ to $\cc{\mc{R}}\bksl W$ along their intersection $\pd W_i\cap\pd(\cc{\mc{R}}\bksl W)$ for each $i\in\{1,\dots,k\}$. Repeating this process for each $b\in B$, we produce a topological space $\cc{\mc{S}}$ and a quotient map $\grj:\cc{\mc{S}}\to \cc{\mc{R}}$ such that $\grj^{-1}(B)$ is finite and $\grj$ sends $\cc{\mc{S}}\bksl \grj^{-1}(B)$ homeomorphically onto $\cc{\mc{R}}\bksl B$. Because of this, we define $\mc{S}=\grj^{-1}(\mc{R})$ and equip $\mc{S}$ with the structure of a Riemann surface. Since the boundary of $\mc{R}$ near each point in $\pd\mc{R}\bksl B$ is given by a simple smooth curve, we can make $\mc{S}$ into a Riemann surface with boundary once we provide a boundary chart at each point of $\grj^{-1}(B)$. To do this, let $b,W_1,\dots,W_k$, and $W$ be as before. For a given $i$, let $x_i$ be the point in $\grj^{-1}(b)$ corresponding to the connected component $W_i$. As $W_i$ is biholomorphic to a simply connected subset of $\CC$, there is a biholomorphic map $\phi$ from $W_i$ onto the open upper half-disc $\DD^+$ that sends $W\cap \pd W_i$ onto the interval $(-1,1)$ and $b$ to $0$. We then take $\phi\of \grj$ to the boundary chart at $x_i$. By \cite[\S I.13.H]{AlSar}, the surface $\mc{S}$ can be embedded into a compact (closed) Riemann surface, the \textit{double} of $\mc{S}$, thus making $\mc{S}$ into a finite Riemann surface.

We define $h:\cc{\mc{S}}\to \mc{V}\cap\cc{\DD^n}$ to be the composition $h=H\of\grj$. Then $h$ is a proper finite continuous map, and $h$ is a diffeomorphism over $\mc{V}^*\cap\cc{\DD^n}$. We also see from \ref{NUE:BndryR} that $\pd\mc{S}=h^{-1}(\mc{V}\cap\TT^n)$, as desired.

\subsection{Proof of Theorem \ref{FinalSec:MainThm}}

Let $\grW$ be a domain in a Riemann surface so that $\cc{\grW}$ is compact. We denote by $A(\grW)$ the algebra of functions continuous on $\cc{\grW}$ and analytic on $\grW$. With $h:\cc{\mc{S}}\to\mc{V}\cap\cc{\DD^n}$ as above, we denote by $A_h(\mc{S})$ the closed unital subalgebra of $A(\mc{S})$ generated by the coordinate functions $h_1,\dots,h_n$ of $h$

As discussed in the beginning section 3 of \cite{AKM}, it follows from \cite{Hyperbol} that there is a finite subset $Y$ of $\mc{S}$ such that any function in $A(\mc{S})$ which vanishes to sufficiently high order on $Y$ also belongs to $A_h(\mc{S})$. It follows from this that $A_h(\mc{S})$ has finite codimension in $A(\mc{S})$, and if $g\in A_h(\mc{S})$ vanishes to sufficiently high order on $Y$, then $g\cdot A(\mc{S})\sbse A_h(\mc{S})$. In particular, there exists a one variable polynomial $Q$ so that $Q(h_n)\cdot A(\mc{S})\sbse A_h(\mc{S})$.

When $\grW$ is a finite Riemann surface and $\tau$ is a finite positive Borel measure supported on $\pd\grW$, we denote by $A^2(\grW,\tau)$ the $L^2(\tau)$-closure of $A(\grW)$. When $\grW=\mc{S}$, we write instead $A^2(\tau)$ and define $A_h^2(\tau)$ to be the $L^2(\tau)$-closure of $A_h(\mc{S})$. Given a bounded Borel function $f$ on $\pd\mc{S}$, we denote by $M_{f,\tau}$ and $N_{f,\tau}$ the operators for multiplication by $f$ on $A^2(\tau)$ and $A_h^2(\tau)$, respectively. If $F=(f_1,\dots,f_n)$ is an $n$-tuple of bounded Borel functions on $\pd\mc{S}$, then we set $M_{F,\grk}=(M_{f_1,\grk},\dots,M_{f_n,\grk})$ and $N_{F,\grk}=(N_{f_1,\grk},\dots,N_{f_n,\grk})$.

\begin{corollary}\label{NUE:AtmlsLem1}
  Let $\tau$ be a diffuse finite positive measure on $\pd\mc{S}$. If $M_{h,\tau}$ is an $n$-tuple of shifts with finite multiplicity, then the same holds for $N_{h,\tau}$, and $M_{h,\tau}\approx N_{h,\tau}$.
\end{corollary}
\begin{proof}
    Since $\tau$ has no atoms, $N_{h_n,\tau}$ has no eigenvalues. Indeed, if $N_{h_n,\tau}f=\grl f$ for some $f\in A^2(\tau)$ and $\grl\in\TT$, then $\int |h_n-\grl|^2 |f|^2d\tau=0$. Since $h_n-\grl$ has only finitely many zeros in $\cc{\mc{S}}$, it follows that $f=0$ $\tau$-a.e. Thus, as $A^2_h(\tau)$ has finite codimension in $A^2(\mu)$, the corollary now follows from Lemma \ref{NUE:PrecLem}.
\end{proof}

Let $\grW$ is a finite Riemann surface, and denote by $u_f$ the harmonic function on $\grW$ with boundary values given by $f\in C(\pd\grW)$. For each $a\in \grW$ there is a Borel probability measure $\grw_a$ on $\pd\grW$ such that
\[ u_f(a)=\int_{\pd\grW} f d\grw_a. \]
In particular, when $v$ is a continuous function on $\cc{\grW}$ that is harmonic on $\grW$, we have $v(a)=\int_{\pd\grW} vd\grw_a$. By \cite{Wermer}, $A(\grW)$ is a hypo-Dirichlet algebra, implying in particular that the set $\{\log|g|:g\text{ invertible in }A(\grW)\}$ has dense span in $C(\pd\grW)$. Thus the measure $\grw_a$ is unique; it is called the \textit{harmonic measure} of $\grW$ corresponding to $a$. Note that $\grw_a$ is a representing measure for the character $f\mapsto f(a)$ on $A(\grW)$, and if $f\in A(\grW)$ is invertible, then
\[ \log|f(a)|=\int_{\pd\grW}\log|f|d\grw_a; \]
that is, $\grw_a$ is an \textit{Arens-Singer} measure.

As seen in \cite{SchifSpen}, every finite Riemann surfaces posses a Green's function $G_a$ for the point $a$, and an application of Green's formula shows that $d\grw_a$ is equal to $-*dG_a$, where $*$ is the conjugation operator. The following two statements concerning harmonic measure on finite Riemann surfaces seem to be well known, but we lack a good reference and thus sketch a proof of each.
\begin{enumerate}
  \item[\rm{(1)}] Let $b\in \pd\grW$, $U$ a coordinate neighborhood about $b$, and $\phi:U\to\DD$ a chart sending $b$ to $0$. Since $\pd\grW$ is smooth, we suppose $U\cap\pd\grW$ is sent smoothly onto the interval $(-1,1)$. If $E\sbse U\cap\pd\grW$, then $\grw(E)=0$ if and only if $\phi(E)$ has Lebesgue measure 0.
  
  To prove this, we modify the argument found in \cite[\S II.2]{GarMar}. We suppose $\phi$ sends $U\cap\grW$ onto the upper half-disc $\DD^+$ and $U$ is small enough that $a\nin U$. The push-forward of $-*dG_a$ from $U\cap\pd\grW$ to $(-1,1)$ is given by $\frac{\pd(G_a\of\phi^{-1})}{\pd y}dx$, where $x+iy$ is the coordinate function on $\DD$. Since $G_a\of\phi^{-1}$ is harmonic on $\DD^+$ and zero on $(-1,1)$, there is a real valued harmonic function $u$ on $\DD$ given by $u(z)=G_a(\phi^{-1}(z))$ when $\Im(z)\geq 0$ and by $u(z)=-G_a(\phi^{-1}(\cc{z}))$ when $\Im(z)<0$. There is then an analytic function $f$ on $\DD$ with $\Im(f)=u$ and $f(0)=0$. Since $f$ sends $\DD^+$ into $\DD^+$, it follows that $(\pd u/\pd y)(0)=f'(0)>0$. A similar argument may be carried out for any point in $(-1,1)$, and therefore $\int_{\phi(E)} \frac{\pd(G_a\of\phi^{-1})}{\pd y}dx=0$ if and only if $\phi(E)$ has Lebesgue measure 0.
  
  \item[\rm{(2)}] If $a,b\in \grW$, then $\grw_a$ and $\grw_b$ are mutually absolutely continuous. Furthermore, there is a number $c>1$ (depending on $a$ and $b$) so that
  \[ \frac{1}{c}<\frac{d\grw_a}{d\grw_b}<c. \]
  
  Since $\pd\grW$ is compact, it suffices to verify these statements locally. Let $\phi$ and $U$ be as above. By (1), $(\grw_a|U)\of \phi^{-1}$ and $(\grw_b|U)\of \phi^{-1}$ are both mutually absolutely continuous with respect to Lebesgue measure on $(-1,1)$. In particular, we see that $d((\grw_a|U)\of \phi^{-1})/dx$, for example, is a strictly positive continuous function on $(-1,1)$.
\end{enumerate}
We remark that (2) can also be proved using Harnack's inequality together with the observation that $a\mapsto \grw_a(E)$ is a harmonic function for any Borel $E\sbse \pd\grW$. The following lemma is now a consequence of (1).

\begin{lemma}\label{NUE:MeasZeroLem}
    Let $\grw$ be harmonic measure for a point in $\mc{S}$ and $E$ a Borel subset of $\pd\mc{S}$. If $\grw(E)=0$, then $h_n(E)\sbse \TT$ has Lebesgue measure 0.
\end{lemma}
\begin{proof}
    We assume $E$ is contained in the domain of a single boundary chart $\phi$. Since $h_n$ is smooth except possibly at a finite number of points, the map $h_n\of\phi^{-1}$ determines a piecewise smooth map from an interval of $\RR$ into $\TT$. As $\grw(E)=0$, it follows from the preceding remarks that $\phi(E)$ has Lebesgue measure 0, and thus $h_n\of\phi^{-1}$ sends $\phi(E)$ to a set in $\TT$ of Lebesgue measure 0.
\end{proof}

\begin{lemma}\label{NUE:HarmIsomLem}
  Let $\grw$ be harmonic measure for a point $x_0\in\mc{S}$. Then $M_{h,\grw}$ and $N_{h,\grw}$ are both $n$-tuples of shifts of finite multiplicity, and
  \begin{equation*}
      \Ann(M_{h,\grw})=\Ann(N_{h,\grw})=\mc{I}.
  \end{equation*}
\end{lemma}
\begin{proof}
  For each $i$, we define $g_i=(h_i-h_i(x_0))/(1-\cc{h_i(x_0)}h_i)$. Then $M_{g,\grw}$ is a completely non-unitary $n$-tuple if and only if $M_{h,\grw}$ is also. We fix $f\in\bigcap_{\gra\in\NN_0^n}M_{g,\grw}^\gra A^2_h(\grw)$ and note that for each $\gra\in\NN^n_0$ there is an $f_\gra\in \Hil$ so that $f=g^\gra f_\gra$. Thus, when $\gra,\grb\in\NN^n_0$ such that $\gra-\grb\in \NN^n$,
  \[ \inner{f}{g^\grb}=\inner{g^{\gra-\grb}f_\gra}{1}=g(x_0)^{\gra-\grb}f_\gra(x_0)=0. \]
  As $\{g^\gra:\gra\in\NN_0^n\}$ is total in $A^2_h(\grw)$, it follows that $f=0$. Thus $M_{h,\grw}$ is completely non-unitary.
  
  Since $\mc{I}\sbse \Ann(M_{h,\grw})$ and $\mc{I}$ is prime, it follows from Lemma \ref{CharAnn:Prime} that $\mc{I}=\Ann(M_{h,\grw})$. Because $M_{h_i,\grw}$ is not multiplication by a scalar for any $i$, Theorem \ref{PrmIdl:MainThm} implies that $M_{h,\grw}$ is an $n$-tuple of shifts, and thus, as $M_{h,\grw}$ is cyclic, we have by Theorem \ref{FinMult:Equiv} that $M_{h_j,\grw}$ has finite multiplicity for each $j$. The lemma now follows from Corollary \ref{NUE:AtmlsLem1}.
\end{proof}

Let $\grW$ be a finite Riemann surface with harmonic measure $\grw$. It is shown in \cite[Lem. 3.9]{AKM} that if $u$ is log-integrable with respect to $\grw$, then there exists an $f\in A^2(\grW,\grw)$ so that $u=|f|^2$ and the closure of $A(\grW)f$ has finite codimension in $A^2(\grW,\grw)$. We remark that while this result is stated in \cite{AKM} for the desingularization of a distinguished variety in $\CC^2$, their argument is carried in the case we have just described. The following lemma is based on \cite[Lem. 3.4]{AKM}.

\begin{lemma}\label{NUE:KKLemma}
  Suppose $\nu$ is a diffuse finite positive measure on $\pd\mc{S}$ and $\grw$ is harmonic measure for $\mc{S}$. If $M_{h,\nu}$ is an $n$-tuple of shifts of finite multiplicity, then $M_{h,\nu}\approx M_{h,\grw}$.
\end{lemma}
\begin{proof}
  Suppose $\nu$ is absolutely continuous with respect to harmonic measure $\grw$, and $\int \log(d\nu/d\grw)d\grw>-\infty$. As we have just seen, there then exists an $f\in A^2(\grw)$ so that $|f|^2=d\nu/d\grw$ and the $L^2(\grw)$-closure $\mc{K}$ of $A(\mc{S})f$ has finite codimension in $A^2(\grw)$. We denote by $\mc{K}_0$ the $L^2(\grw)$-closure of $A_h(\mc{S})f$ and note that $\mc{K}_0$ has finite codimension in $\mc{K}$. Thus
  \[ M_{h,\grw}|\mc{K}_0=N_{h,\grw}|\mc{K}_0\lesssim N_{h,\grw}|\mc{K} \lesssim N_{h,\grw}. \]
  For $g_1,g_2\in A_h(\mc{S})$ we have
  \[ \inner{g_1 f}{g_2f}=\int g_1\cc{g_2}|f|^2d\grw=\int g_1\cc{g_2}d\nu. \]
  Thus $M_{h,\grw}|\mc{K}_0$ is unitarily equivalent to $M_{h,\nu}$, and so $M_{h,\nu}\lesssim N_{h,\grw}$. Lemma \ref{NUE:PrecLem} now implies that $M_{h,\nu}\approx N_{h,\grw}$, and so $M_{h,\nu}\approx M_{h,\grw}$ by Lemma \ref{NUE:AtmlsLem1}. To prove the lemma, it therefore suffices to show that $\nu$ is absolutely continuous with respect to $\grw$ and $\log(d\nu/d\grw)\in L^1(\grw)$.
  
  We first show that $\nu$ is absolutely continuous with respect to harmonic measure $\grw$ for a point $x_0\in\mc{S}$. By Corollary \ref{NUE:AtmlsLem1}, $N_{h,\nu}$ is an $n$-tuple of shifts of finite multiplicity. Denote by $\phi$ the character of $A(\mc{S})$ given by evaluation at $x_0$. As mentioned before, $A(\mc{S})$ is a hypo-Dirichlet algebra and $\grw$ is an Arens-Singer measure. Thus, by the first corollary of \cite[Thm. 3.1]{AhernSar}, every representing measure for $\phi$ is absolutely continuous with respect to $\grw$. We denote by $\nu_s$ the part of $\nu$ that is singular with respect to $\grw$, and note that there is an $F_\grs$-subset $E$ of $\pd\mc{S}$ such that $\grw(E)=0$ and $\nu_s(E^c)=0$. By Forelli's Lemma (see \cite[Lem. II.7.3]{Gamelin}), the characteristic function $\chi_E$ of $E$ is contained in $A^2(\nu)$. We set $g=(h_n-h_n(x_0))/(1-\cc{h_n(x_0)}h_n)$, $\tau=\nu_s\of g^{-1}$, and denote by $\mc{H}_0$ the cyclic subspace of $A^2(\nu)$ generated $N_{g,\nu}$ and $\chi_E$. For any $p,q\in\CC[X_n]$,
  \[ \inner{p(g)\chi_E}{q(g)\chi_E}=\int_{\pd\mc{S}}(p\of g)\cc{q\of g}d\nu_s=\int_\TT p(w)\cc{q(w)}d\tau(w), \]
  where we have used the fact that $g(\pd\mc{S})=\TT$. Thus the algebra generated by $N_{g,\nu}|\mc{H}_0$ is unitarily equivalent to the disc algebra $A(\DD)$ acting on its $L^2(\tau)$-closure. It follows from Lemma \ref{NUE:MeasZeroLem} that $g(E)\sbse \TT$ has Lebesgue measure 0, and we easily see that $\tau(g(E)^c)=0$. That is, $\tau$ is singular with respect to Lebesgue measure on $\TT$. It now follows from the Kolmogorov-Krein theorem on the disc that
  \[ \inf_{f\in A(\DD)}\int_\TT |1-wf(w)|^2 d\tau(w) =0. \]
  Thus $N_{g,\nu}\mc{H}_0=\mc{H}_0$, implying that $N_{h_n,\nu}|\mc{H}_0$ is a unitary operator and so $\mc{H}_0=\{0\}$. In particular, $\nu_s(E)=\|\chi_E\|^2=0$ and therefore $\nu$ is absolutely continuous with respect to $\grw$.
  
  Suppose there is a harmonic measure $\grw_1$ for some point in $\mc{S}$ so that we have $\int \log(d\nu/d\grw_1)d\grw_1>-\infty$. For any other choice of harmonic measure $\grw_2$, there is a constant $c>1$ so that $1/c<d\grw_1/d\grw_2<c$, and thus
  \[ \int\bigg|\log\left(\frac{d\nu}{d\grw_2}\right)\bigg|d\grw_2\leq c\int\bigg|\log\left(\frac{d\nu}{d\grw_1}\right)\bigg|d\grw_1+|\log(c)|. \]
  It therefore suffices to show that $\int \log(d\nu/d\grw)d\grw>-\infty$ for just one choice of harmonic measure $\grw$.
  
  Let $Q\in\CC[X_n]$ such that $(Q\of h_n)A(\mc{S})\sbse A_h(\mc{S})$ and let $z_n\in\DD$ satisfy $Q(z_n)\neq 0$. By Lemma \ref{NUE:NotInRange}, there is a $z\in Z(\mc{I})\cap\DD^n$ so that
  \[ \inf_{p\in\mc{J}}\int |1-p\of h|^2 d\nu >0 \]
  where $\mc{J}$ is the ideal in $\CC[\XX]$ generated by $X_1-z_1,\dots,X_n-z_n$. Choose $x_0$ so that $z=h(x_0)$ and let $\grw$ denote the corresponding harmonic measure. We again denote by $\phi$ the character on $A(\mc{S})$ given by evaluation at $x_0$. In order to show that $\int \log(d\nu/d\grw)d\grw>-\infty$ for some $\grw$, it suffices to show, according the corollary of \cite[Thm. 10.1]{AhernSar}, that $\inf_{f\in\ker\phi}\int|1-f|^2d\nu>0$. We note that $\ker(\phi)\cap A_h(\mc{S})\spse (Q\of h_n)\ker\phi$, and so
  \begin{eqnarray*}
      \inf_{f\in A_h(\mc{S})\cap\ker\phi}\int |Q\of h_n-f|^2 d\nu & \leq & \inf_{f\in \ker\phi}\int |1-f|^2 |Q\of h_n|^2 d\nu \\
      & \leq & M\cdot \inf_{f\in\ker \phi}\int |1-f|^2d\nu,
  \end{eqnarray*}
  where $M=\sup_{w\in\TT}|Q(w)|^2$. As $\CC[h_1,\dots,h_n]$ is dense in $A_h(\mc{S})$, it follows that $\mc{J}(h)$ is dense in $A_h(\mc{S})\cap\ker\phi$. Since $Q\of h_n-Q(z_n)\in A_h(\mc{S})\cap\ker \phi$, we have
  \[ |Q(z_n)|^2\cdot \inf_{p\in\mc{J}} \int|1-p\of h|^2 d\nu \leq M\cdot \inf_{f\in\ker\phi} \int |1-f|^2 d\nu. \]
  Thus $\inf_{f\in\ker\phi}\int|1-f|^2d\nu>0$.
\end{proof}

\begin{proof}[Proof of Theorem \ref{FinalSec:MainThm}]
    By Lemma \ref{NUE:ApplyDecomp}, it suffices to consider the case where $\mc{I}=\Ann(\VV)$ is prime. By comments at the beginning of \S \ref{DesingSubsec}, we may (and do) suppose that $\VV=M_{h,\nu}$ for some diffuse finite positive Borel measure $\nu$ on $\pd\mc{S}$. The theorem now follows from Lemma \ref{NUE:KKLemma}.
\end{proof}



{\small
\begin{thebibliography}{99}
    \bibitem{AKM} Agler, J.; Knese, G.; McCarthy, J. E. \textit{Algebraic pairs of isometries}, J. Operator Theory 67 (2012), no. 1, 215-236.
    \bibitem{AgMc2005} Agler, J.; McCarthy, J.E., \textit{Distinguished varieties}, Acta Math. 194 (2005), 133-153.
    \bibitem{Hyperbol} Agler, Jim; McCarthy, John E. \textit{Hyperbolic Algebraic and Analytic Curves}. Indiana Univ. Math. J. 56 (2007), no. 6, 2899-2933.
    \bibitem{AhernSar} Ahern, P. R.; Sarason, D., \textit{The $H^p$ spaces of a class of function algebras}. Acta Math. 117 (1967) 123-163.
    \bibitem{AlSar} Ahlfors, L.; Sario, L., \textit{Riemann Surfaces}. Princeton Mathematical Series, No. 26 Princeton University Press, Princeton, NJ 1960.
    \bibitem{BDF} Bercovici, H.; Douglas, R.G.; Foias, C. \textit{Canonical models of bi-isometries}. Oper. Theory Adv. Appl., 218 (2012), 177-205.
    \bibitem{Chirka} Chirka, E. M.; \textit{Complex Analytic Sets}, Translated from the Russian by R. A. M. Hoksbergen. Mathematics and its Applications (Soviet Series), 46. Kluwer Academic Publishers Group, Dordrecht, 1989.
    \bibitem{CoxEtAl} Cox, D.; Little, J.; O'Shea, D. \textit{Ideals, Varieties, and Algorithms}, 2nd edition, Springer-Verlag, New York, 1996.
    \bibitem{Gamelin} Gamelin, T. W. \textit{Uniform Algebras}. Prentice-Hall, Inc., Englewood Cliffs, N.J., 1969.
    \bibitem{GarMar} Garnett, J. B.; Marshall, D. E., \textit{Harmonic measure}. New Mathematical Monographs, 2. Cambridge Univ. Press, Cambridge, 2005.
    \bibitem{Ito} It\^{o}, T. \textit{On the commutative family of subnormal operators.} J. Fac. Sci. Hokkaido Univ. Ser. I 14 (1958) 1-15.
    \bibitem{Kendig} Kendig, K. \textit{Elementary algebraic geometry}. Graduate Texts in Mathematics, No. 44. Springer-Verlag, New York-Berlin, 1977.
    \bibitem{Kollar} Koll\'{a}r, J. \textit{Lectures on resolution of singularities}. Annals of Mathematics Studies, 166. Princeton Univ. Press, Princeton, NJ, 2007.
    \bibitem{Popovici} Popovici, D., \textit{A Wold-type decomposition for commuting isometric pairs}, Proc. Amer. Math. Soc. 132 (2004) 2303-2314.
    \bibitem{NFBK} Sz.-Nagy, B., Foias, C., Bercovici, H., K\'{e}rchy, L., \textit{Harmonic Analysis of Operators on Hilbert Spaces, Second Edition}, Springer-Verlag, New York, 2010.
    \bibitem{SchifSpen} Schiffer, M.; Spencer, D.C., \textit{Functionals of finite Riemann surfaces}. Princeton Univ. Press, Princeton, N.J., 1954.
    \bibitem{Suciu} Suciu, I., \textit{On the semigroups of isometries}, Studia Math. 30 (1968), 101-110.
    \bibitem{Vasilescu} Vasilescu, F., \textit{Analytic functional calculus and spectral decompositions}. Translated from the Romanian. Mathematics and its Applications (East European Series), 1. D. Reidel Publishing Co., Dordrecht; Editura Academiei Republicii Socialiste Rom\^{a}nia, Bucharest, 1982.
    \bibitem{Wermer} Wermer, J., \textit{Maximal ideal spaces}. Amer. J. Math. 86 (1964) 161-170.
\end{thebibliography}
}
\end{document}